\documentclass{article}

\usepackage{arxiv}

\usepackage[utf8]{inputenc} 
\usepackage[T1]{fontenc}    
\usepackage{hyperref}       
\usepackage{url}            
\usepackage{booktabs}       
\usepackage{amsfonts}       
\usepackage{nicefrac}       
\usepackage{microtype}      
\usepackage{amsmath,amssymb,amsthm,graphicx,url,multirow,verbatim}
 \usepackage{url}
\usepackage{subcaption}

\usepackage{listings}
\usepackage{hyperref}
\usepackage{soul,environ}
\usepackage{mathtools}
\usepackage{graphics}
\usepackage[ruled, vlined]{algorithm2e}
\newtheorem{definition}{Definition}
\newtheorem{proposition}{Proposition}
\newtheorem{example}{Example}
\newtheorem{theorem}{Theorem}
\newtheorem{lemma}{Lemma}
\newtheorem{corollary}{Corollary}
\newtheorem{result}{Result}

\mathtoolsset{showonlyrefs}

\newcommand{\RR}{\mathbb{R}}	
\newcommand{\Lll}{\ensuremath{\Lambda}}	
\newcommand{\lii}{\ensuremath{\Lambda}}	
\newcommand{\liu}{\ensuremath{\Lambda^*}}	
\newcommand{\Llu}{\ensuremath{{\Lambda}^*}}	
\newcommand{\lru}{\ensuremath{{\Lambda}_r^*}}	

\newcommand{\sn}{\ensuremath{s_n^+}}	
\newcommand{\snn}{\ensuremath{S_n^+}}	
\newcommand{\snu}{\ensuremath{{(S_n^+)}^*}}	
\newcommand{\sm}{\ensuremath{s_m^-}}	
\newcommand{\smm}{\ensuremath{S_m^-}}	
\newcommand{\smu}{\ensuremath{{(S_m^-)}^*}}	
\newcommand{\sqq}{\ensuremath{S_q^+}}	

\newcommand{\Dym}{\ensuremath{\Delta_{mi}}}	
\newcommand{\Dyk}{\ensuremath{\Delta_{m\ih}}}	
\newcommand{\Dxn}{\ensuremath{\nabla_{ni}}}	
\newcommand{\Dxk}{\ensuremath{\nabla_{n\ih}}}	

\newcommand{\NM}{\ensuremath{\Pi}}	
\newcommand{\Nn}{\ensuremath{\Pi_{N}}}	
\newcommand{\NMu}{\ensuremath{\Pi^u}}	
\newcommand{\Nnu}{\ensuremath{\Pi_{N}^u}}	

\newcommand{\Ax}{\ensuremath{\alpha_{1}}}	
\newcommand{\Axx}{\ensuremath{\alpha_{2}}}	
\newcommand{\Axn}{\ensuremath{\alpha_{N}}}	
\newcommand{\By}{\ensuremath{\beta_{1}}}	
\newcommand{\Byy}{\ensuremath{\beta_{2}}}	
\newcommand{\Bym}{\ensuremath{\beta_{M}}}	

\newcommand{\map}{\ensuremath{\mathbf{\mathit{m}}}}
\newcommand{\Uj}{\ensuremath{\mathcal{U}_j}}	
\newcommand{\U}{\ensuremath{\mathcal{U}}}	
\newcommand{\Ei}{\ensuremath{\mathcal{E}^{i}(\mathcal{U})}}	
\newcommand{\Eih}{\ensuremath{\mathcal{E}^{\ih}(\mathcal{U})}}
\newcommand{\un}{\ensuremath{\sigma}} 

\newcommand{\minU}{\ensuremath{\upsilon_{\ih}^*}} 
\def\ih{\hat{\imath}}
\newcommand{\Ti}{\ensuremath{\theta^{i}}}		
\newcommand{\Tih}{\ensuremath{\theta^{\ih}}}		
\newcommand{\Tii}{E^{i}}	

\newcommand{\Nx}{\ensuremath{x}}
\newcommand{\NX}{\ensuremath{X}}
\newcommand{\Ny}{\ensuremath{y}}
\newcommand{\NY}{\ensuremath{Y}}

\newcommand{\UX}{\ensuremath{\bar{X}}}

\newcommand{\UY}{\ensuremath{\bar{Y}}}
\usepackage{natbib}
 \bibpunct[, ]{(}{)}{,}{a}{}{,}%

\usepackage{url}

\title{Uncertain Data Envelopment Analysis: Box Uncertainty}

\author{
  Emma~Stubington\thanks{Corresponding author} \\
 STOR-i Centre for Doctoral Training\\ Lancaster University\\ Lancaster, UK\\
   \texttt{e.stubington@lancaster.ac.uk} \\
   \And
 Matthias~Ehrgott \\
  Department of Management Science\\ Lancaster University\\Lancaster, UK\\
  \texttt{m.ehrgott@lancaster.ac.uk} \\
   \And
Omid~Nohadani \\
  Department of Industrial Engineering\\ Northwestern University\\ Evanston, IL, USA\\\texttt{nohadani@northwestern.edu} \\
}

\begin{document}
\maketitle

\begin{abstract}
Data Envelopment Analysis (DEA) is a nonparametric, data driven technique used to perform relative performance analysis among a group of comparable decision making units (DMUs). Efficiency is assessed by comparing input and output data for each DMU via linear programming. Traditionally in DEA, the data are considered to be exact. However, in many real-world applications, it is likely that the values for the input and output data used in the analysis are imprecise. To account for this, we develop the uncertain DEA problem for the case of box uncertainty. We introduce the notion of DEA distance to determine the minimum amount of uncertainty required for a DMU to be deemed efficient. For small problems, the minimum amount of uncertainty can be found exactly, for larger problems this becomes computationally intensive. Therefore, we propose an iterative method, where the amount of uncertainty is gradually increased. This results in a robust DEA problem that can be solved efficiently. This study of uncertainty is motivated by the inherently uncertain nature of the radiotherapy treatment planning process in oncology. We apply the method to evaluate the quality of a set of prostate cancer radiotherapy treatment plans relative to each other.
\end{abstract}


\section{Introduction and Motivation}\label{Sec:Intro}

Standard Data Envelopment Analysis (DEA) models assume that the input and output data are exact and decision making units (DMUs) are classified as inefficient or efficient based on these data. However, the data of many real-world applications are inherently uncertain. Hence, it is possible that an inefficient DMU performs well in practice, i.e. it is the uncertainty in the data that stops it being classified as efficient. The reliability of the conclusions from a DEA model is intrinsically linked to the quality of the data used and hence methods to account for uncertainty are required. In this paper, we study the uncertain nature of the data and the effect this has on an individual DMU's efficiency score. We follow the robust optimisation framework and assume the input and output data are in fact realisations from a range, called an uncertainty set. An uncertainty set comprises of all the possible values that the uncertain input and output data can take. The nature of the uncertainty set is selected by the user to capture the uncertainty in the problem, \citet{gorissen2015practical}. Robust optimisation has become popular because it provides modelling flexibility while considering the implications of uncertain data. 

Building upon \citet{ehrgott2018uncertain}, we refine the concept of uncertain DEA (uDEA) for the specific case of box uncertainty. Solving a uDEA model determines the maximum efficiency score that a DMU can achieve over all admissible uncertainties and the minimum amount of uncertainty required to achieve this efficiency score. This refinement is motivated by the radiotherapy treatment planning application which we discuss in Section \ref{SubSec:CaseStudy}. 

\subsection{Contributions of this paper}
Box uncertainty is one of the most widely used uncertainty sets and for uncertain linear optimisation problems results in a tractable\footnote{Here we will use the notion of a tractable solution to describe a problem that can be reformulated into an equivalent problem that can be solved in polynomial time.} robust counterpart. This allows us to develop a good approximation to the uDEA problem as introduced in \citet{ehrgott2018uncertain}. In Section \ref{SubSec:DEADistance}, we define the notion of DEA distance. This allows us to investigate the relationship between the amount of uncertainty required for an inefficient DMU to be deemed efficient. In Section \ref{SubSec:ChangeU}, we propose an iterative method to solve the uDEA problem for box uncertainty, where the amount of uncertainty is gradually increased. This results in a robust DEA problem that can be solved efficiently. The developed method is then used to compare radiotherapy treatment plans for prostate cancer. At the same time, this accounts for the numerous sources of uncertainty arising in the planning and delivery of radiotherapy treatment plans.

\section{Uncertain Data Envelopment Analysis}

\subsection{General Data Envelopment Analysis}

Throughout, we assume the (standard) input oriented envelopment BCC (variable returns to scale) DEA model, but the following theorems and concepts can easily be adapted to other DEA formulations. We will refer to this as the nominal DEA problem. Consider $I$ DMUs, each with $M$ outputs and $N$ inputs, $I > M+N$. The efficiency score $E^{i}$ for DMU $i$ is defined by solving the linear program
\begin{equation} \label{Eq:Nominal}
E^{i} = \min \{ \Ti : \NY \lambda - \Ny^{i} \geq 0, \;
        \NX \lambda - \Ti \Nx^{i} \leq 0, \; e^T \lambda = 1, \; \lambda \geq 0 \},
\end{equation}
where $e\in \RR^I$ is the vector of ones. The nominal data matrices $\NY$ and $\NX$ are the non-negative, non-zero $M \times I$ output and $N \times I$ input matrices, such that $\NY_{mi}(\NX_{ni})$ is the $m(n)$\textsuperscript{th} output(input) value for DMU $i$. We denote the $m(n)$\textsuperscript{th} row of $\NY(\NX)$ by $\NY_m(\NX_n)$ and the $i$-th column by $\Ny^{i}$ and $\Nx^{i}$. Throughout, we will use the notation $m \in \mathcal{M},~n\in \mathcal{N}$ and $i \in \mathcal{I}$, where $\mathcal{M}=\{1,\ldots M\},~\mathcal{N}=\{1,\ldots N\}$ and $~\mathcal{I}=\{1,\ldots I\}$, respectively. 
The solution $(\lambda,\Ti)=(0,0, \dots, 1,\allowbreak \dots ,0,1)$, where the first one is in the $i$\textsuperscript{th} position is feasible for \eqref{Eq:Nominal}. The feasibility of this solution ensures the maximum value that $\Ti$ can take is $1$, i.e. an optimal value $E^{i}$ of model \eqref{Eq:Nominal} is less than or equal to one. In this way, we denote an optimal solution for DMU $i$ by $( \Lambda,E^{i})$, $\Lambda \in \RR^I.$ We use the following definition of efficiency.

\begin{definition}\label{Def:EffDMU}
DMU $i$ is efficient if and only if $E^{i}=1$.
\end{definition}
Therefore, if $E^{i}<1$ DMU $i$ is inefficient. To find the efficiency scores for the $I$ DMUs \eqref{Eq:Nominal} must be solved $I$ times, once for each of the $I$ DMUs. 

\begin{definition}\label{Def:PPS}
 The production possibility set, PPS, is defined to be
 \begin{equation}
  T=\{(x,y) \in \RR^N \times \RR^M: 
      \; x \ge \NX \lambda,\; y \le \NY \lambda, \; e^T \lambda = 1, \; \lambda \geq 0\}.
   \end{equation}
\end{definition}

\begin{definition}
 The efficient frontier is the subset of points of set $T$ satisfying the efficiency condition from Definition \ref{Def:EffDMU}.
\end{definition}
In DEA, the efficient frontier represents the standard of performance that all DMUs should try to achieve. DMUs that are inefficient, $\Ti<1$, are projected to the efficient frontier such that the proportional reduction in inputs is maximised while the outputs of DMU $i$ remain fixed.
\begin{definition}\label{Def:ExtremePoint}
 $(\Nx,\Ny)\in T$ is an extreme point of $T$ if and only if there are no points 
 $(x',y'),~(x'',y'') \in T$ and no $\alpha \in (0,1)$ such that
 \begin{equation} (x,y)=\alpha (x',y') +  (1-\alpha) (x'',y'').
 \end{equation} 
\end{definition}

All extreme points of $T$ are efficient DMUs. 
Throughout, we use $i \in \mathcal{I}$ to represent a generic DMU, efficient or inefficient, and $\ih$ to denote the (inefficient) DMU under consideration.

\subsection{Uncertain DEA}

In \eqref{Eq:Nominal}, the data $\NX$ and $\NY$ are assumed to be exact. However, in many real-world applications this is not the case. Therefore, we consider the effect that uncertainty has via uDEA, as introduced in \citet{ehrgott2018uncertain}. To establish a framework for solving the uDEA problem for the specific case of box uncertainty, we introduce some key definitions from \citet{ehrgott2018uncertain}. 
Following the robust optimisation methodology, we assume that the uncertainty can be modelled constraint-wise. In this way, each constraint from \eqref{Eq:Nominal} is replaced with a set of constraints and we define the robust efficiency score for DMU $i$, $\Ei$, to be the optimal value of the robust DEA model,
\begin{equation} \label{Eq:Robust}
\begin{split}
\Ei = \min \{ &\Ti : \UY_j \lambda - \UY_{ji} \geq 0, \forall ~\UY_j \in \mathcal{U}_j,~j=1,\ldots, M,\\
              &\UX_{j-M} \lambda - \Ti\UX_{(j-M)i} \leq 0,~\forall~ \UX_{j-M} \in \mathcal{U}_j,~j=M+1,\ldots, M+N,\\
              & e^T \lambda = 1, \lambda \geq 0 \},
\end{split}
\end{equation}
where $\UY(\UX)$ is the uncertain output(input) data and $\mathcal{U}_j$ is the uncertainty set for output(input) $j$. Throughout, uncertain data are distinguished from the certain nominal data by bar notation. Note that we use the notation $j \in \mathcal{J},$ where $\mathcal{J}=\{1,\ldots, M+N\}$. 
\begin{definition}\label{Def:UncertainInOut}
 The uncertain inputs and outputs are defined as follows:
\begin{enumerate}
 \item $\mathcal{U}_j$
  is an uncertainty set that models the possible values of the output(input) data $\UY_j(\UX_{j-M})$. Hence each $\UY_j \in \mathcal{U}_j,~j=1,\ldots, M(\UX_{j-M} \in \mathcal{U}_j,j=M+1,\ldots, M+N)$ is a possible row vector of output(input) data for the $j(j-M)$th output(input).
 \item $\mathcal{U} = \{\mathcal{U}_j: j \in \mathcal{J}\}$ 
 is a collection of uncertainty. Hence $\mathcal{U}$ contains the totality of uncertainty across all inputs and outputs.
\end{enumerate}
\end{definition}

To ensure consistency with DEA methodology we must place a restriction on the nature of the uncertainty sets which is not required in a more general robust optimisation setting. We assume that the introduction of uncertainty does not introduce negative data. 
\begin{definition}\label{def-harbor}
Consider the two collections of uncertainty,
\begin{equation}
\mathcal{U}' = \{ \mathcal{U}'_j : j \in \mathcal{J} \} \; \mbox{ and } \;
\mathcal{U}'' = \{ \mathcal{U}''_j : j \in \mathcal{J} \}.
\end{equation}
We say that $\mathcal{U}''$ harbours at least the uncertainty of $\mathcal{U}'$, 
denoted by $\mathcal{U}' \unlhd \, \mathcal{U}''$, if $\mathcal{U}'_j \subseteq \mathcal{U}''_j$ for 
$j \in \mathcal{J}$.
\end{definition}

For the nominal data $\NY(\NX)$, we have a fixed data collection $\mathcal{U}^0$ where $\mathcal{U}^0_j \in \mathcal{U}^0,~j \in \mathcal{J},$ is such that $\mathcal{U}^0_j=\{\NY_j\},~ j=\{1, \ldots , M\},~\mathcal{U}^0_j=\{\NX_{j-M}\},~j=\{M+1, \ldots ,N+ M\}$. We only consider $\mathcal{U}$ such that $\mathcal{U}^0 \unlhd \mathcal{U},$ i.e. we only consider uncertainty sets which harbour at least the same amount of uncertainty as the nominal data, which is no uncertainty.

In order to model and analyse the effect of uncertainty, we allow the data to come from a variety of uncertainty sets. Formally, we say
$\Omega$ is the universe of possible collections of uncertainty.
\citet{ehrgott2018uncertain} define a numerical association with $\mathcal{U}$ as in Definition \ref{Def:Amount}.

\begin{definition}\label{Def:Amount}
 An amount of uncertainty is a mapping 
 \begin{equation}
    \map:\Omega \rightarrow \RR_+ : \U \mapsto \map (\U) 
 \end{equation} \vspace{-1em} such that
 \begin{enumerate}
 \item there is zero uncertainty if and only if there is no uncertainty i.e. $\map (\U) = 0$ if and only if $| \Uj| = 1$ for $j \in \mathcal{J}$,
 and
\item $\map (\U)$ is monotonic, i.e. $\map (\U') \leq \map(\U'')$ if $\mathcal{U}' \unlhd \, \mathcal{U}''$.
\end{enumerate}
\end{definition}

The above definitions allow for a large amount of modelling flexibility as $\map(\U)$ and $\Omega$ can be modelled differently for each DMU and each output/input and there is no restriction on the uncertainty set. We let $\left( \Omega, \map \right)$ be a configuration of uncertainty, where $\Omega$ is a universe of possible collections of uncertainty satisfying $\mathcal{U}^0 \unlhd \mathcal{U}~ \forall~\mathcal{U}\in \Omega$. As in \citet{ehrgott2018uncertain}, we suggest an assessment of the amount of uncertainty such as
\begin{equation}
   \map(\U)=\| \U \|_{p,q} := \Big\| \langle \|R_1\|_p, \dots, \|R_{M+N}\|_p \rangle \Big\|_q,
\end{equation}
which aligns with the standard $L_{p,q}$ notation associated with matrix norms. Note that in \citet{ehrgott2018uncertain}, the matrices $R_j$ as well as the values of $p$ and $q$ describe the shape and size of the uncertainty sets $\U_j$. Here, we consider the special case of box uncertainty and therefore, we use the following definition.

\begin{definition} \label{Def:BoxUncertainty}
 Box uncertainty is the configuration of uncertainty $\left( \Omega, \map(\U) \right)$ with
 \begin{equation}
  \U_j (\sigma_j) = \begin{cases} \{\NY_j + \sigma_j u^T I^u:||u||_\infty \leq 1\} & j=1,\ldots, M \\
                       \{\NX_{j-M} + \sigma_{j-M} u^T I^u:||u||_\infty \leq 1\} & j=M+1,\ldots, M+N,  \end{cases}
 \end{equation}
 where $I^u$ is the $I\times I$ identity matrix. The universe of possible uncertain collections is
 \begin{equation}
 \begin{aligned}\label{Eq:Omega}
    \Omega 
           = &\{\U_j(\sigma_j):0\leq \sigma_j\leq \nu_j,~j \in \mathcal{J}\},
    \end{aligned}
 \end{equation} 
 and $\nu_j$ is the maximum amount of uncertainty output $j$(input $j-M$) can take and $\map(\U)=||\U||_{\infty,\infty}= \max \{|\sigma_j|:j \in \mathcal{J}\}$, i.e. $R_j = \sigma_j I^u$.
\end{definition}
To ensure consistency with DEA methodology, we assume that the uncertain data is bounded below by zero, i.e. the introduction of uncertainty does not introduce negative data into the model. In Definition \ref{Def:BoxUncertainty}, $\Omega$ is restricted by $\nu$, the maximum amount of acceptable uncertainty. If $\Omega$ is unrestricted, then $\nu=\infty.$ 
In most real world applications, we assume that such an upper bound exists (otherwise we would essentially state that the value of a particular output(input) is completely unknown  and can take any value in $\RR^+_0$). For example, in Section \ref{SubSec:CaseStudy}, an upper bound on the amount of uncertainty is used such that the dose of radiation delivered to the tumour and organs at risk is within clinically acceptable guidelines. 

In Definition \ref{Def:BoxUncertainty}, each data point can vary within $\pm \sigma_j$ from its nominal value. W.l.o.g. we assume that the possible deviation from the nominal data value is the same for each output/input. This is because BCC DEA models are not affected by scaling, \citep{charnes1984preface}. 
Therefore, the outputs and inputs can be scaled so that the amount of uncertainty is the same for each output and input. Throughout, we assume this scaling of the outputs and inputs has already occurred. We use $\NY_{1i}, \dots , \NY_{Mi},\NX_{1i}, \dots ,\NX_{Ni},~i\in \mathcal{I}$ to denote the scaled data and drop the subscript $j$ on $\sigma$.
In this way, $\map(\U)=\sigma$.

For the specific case of box uncertainty \eqref{Eq:Robust} becomes 
\begin{equation} \label{Eq:RobustBoxDEA}
    \begin{aligned}
         \Ei =& \min \{\Ti : (\NY+\Delta) \lambda - (\Ny^{i}+\delta_y^{i}) \geq 0,\\&(\NX+\nabla) \lambda - \Ti (\Nx^{i}+\delta_x^{i}) \leq 0, e^T \lambda = 1,
        \lambda \geq 0\},
    \end{aligned}
\end{equation}
where $\Delta_{mi},\nabla_{ni} \in [-\un,\un]$ for all $i, m, n$. The matrix $\Delta(\nabla)$ is the uncertainty associated with the output(input) matrix such that $\Delta_{mi}(\nabla_{ni})$ is the uncertainty associated with the $m(n)$\textsuperscript{th} output(input) value for DMU $i$. We denote the $m(n)$\textsuperscript{th} row of $\Delta(\nabla)$ by $\Delta_m(\nabla_n)$ and the $i$-th column by $\delta_y^{i}(\delta_x^{i})$. 
An optimal solution to \eqref{Eq:RobustBoxDEA} is denoted $\left( \hat{\Lambda}, \mathcal{E}^{\ih}(\mathcal{U})\right)$. 

\citet{ehrgott2018uncertain} define the uDEA problem for the $i$\textsuperscript{th} DMU to be,
\begin{equation} \label{Eq:uDEA}
\gamma^*_i 
 =  \sup_{0 \leq \gamma_i  \leq 1} \{ \gamma_i  :
	~ \min_{\U \in \Omega} ~ \{ \map(\U) : 
		\Ei \geq \gamma_i  \} \}.
\end{equation}
Here, $(\Lambda^*,\gamma^*_i,\Delta^*,\nabla^*)$, denotes the values the variables take at an optimal solution to the uDEA problem \eqref{Eq:uDEA}. In this way, \eqref{Eq:Robust} is the special case of \eqref{Eq:uDEA} when the amount of uncertainty is fixed. We now wish to distinguish between those DMUs that remain inefficient in the presence of uncertainty and those that can achieve efficiency. To do this, we use the definition from \citet{ehrgott2018uncertain} of capable and incapable DMUs.
\begin{definition}\label{Def-capable} A DMU $i$ under $\Omega$ is:
\begin{enumerate}
\item capable if $\gamma^*_{i} =\Ei=1 $ 
	for some $\U \in \Omega$.
\item weakly incapable if $\gamma^*_{i} = 1$ 
	but $\Ei< 1$ for all $\U \in \Omega$.
\item strongly incapable if $\gamma^*_{i} < 1$.
\item incapable if it is either
	strongly or weakly incapable.
\end{enumerate}
\end{definition}
In this way, a DMU is incapable if it is inefficient for all uncertain data instances, i.e. even with the ability to select the most favourable data from $\Omega$, it has no claim on efficiency. Without changing the operation of the DMU, the only way for an incapable DMU to become capable is if a change in $\Omega$ occurs.

\subsection{Initial Results}

We now introduce some initial results to be used in the following sections. 
To refer to points of the PPS that are defined by uncertain data of existing DMUs, we use the notion from \citet{thanassoulis1998simulating} of unobserved DMUs. Here we call them virtual DMUs and define them in the following way. 

\begin{definition}\label{Def:VirtualDMU}
 The virtual DMU $i^u,~i \in \mathcal{I}$, is defined to be the DMU with inputs and outputs $\left(x^{i}+\un,y^{i}-\un\right),~\forall~ i\neq\ih \in \mathcal{I}$ and $\left(x^{\ih}-\un,y^{\ih}+\un\right)$ for $\ih$. 
\end{definition}

This allows us to consider projections from the point $(x^{i^u},y^{i^u})$ to the efficient frontier. We assume that the PPS has dimension $M+N$ and hence the facet of the efficient frontier DMU $\ih$ is projected to can be defined by at most $\Phi=N+M$ DMUs. A hyperplane in $\Phi$ dimensions can be defined by
\begin{equation}\label{Eq:PlaneN+M}
  \NM = \{r \in \RR^{\Phi}: r^T \eta_1 = d\},
 \end{equation}
where $\eta_1=(\Ax, \dots, \Axn , \By, \dots ,\Bym)^T$ is a normal vector to the plane and $d$ is a constant. We let $\Psi$ denote the set of hyperplanes that define facets of the efficient frontier and let $\NMu$ be the hyperplane formed by the parallel translation $\NM \to \NMu$, where
$(\Nx,\Ny) \in \NM \to (\Nx+\un,\Ny-\un) \in \NMu$. Then $\Psi^u$ is the set of hyperplanes that define facets of the transformed efficient frontier when uncertainty is introduced (this follows from Theorem \ref{Thm:MaxIncUFix} below). When solving \eqref{Eq:Nominal} for inefficient DMU $\ih$, it is projected to a point on the efficient frontier along a trajectory of fixed outputs. This is the unique point the achievement of which will render the DMU efficient on the nominal efficient frontier. However, the introduction of uncertainty in \eqref{Eq:uDEA} means there may be alternate hyperplanes $\NMu \in \Psi^u$ to which DMU $\ih^u$ is projected, establishing an increase in its efficiency score. We wish to determine to which facet of the efficient frontier DMU $\ih^u$ should be projected, such that the least amount of uncertainty is required for DMU $\ih^u$ to be deemed efficient. To describe these alternative projections, we define target points in Definition \ref{Def:Target}.
\begin{definition}\label{Def:Target}
The target point $T(i^u,\NMu)$ is defined to be the point on the plane $\NMu$ to which DMU $i^u$ is projected.
\end{definition} 
Here, we note that our use of projection in Definition \ref{Def:Target} is not a standard DEA term. Definition \ref{Def:Target} allows us to consider the projection of DMU $\ih^u$ to all $\NMu \in \Psi^u$ as opposed to the single projection in the nominal DEA problem \eqref{Eq:Nominal}.

Let $\sm\geq 0$ ($s^{+}_n\geq 0$) be the slack in the $m(n)$\textsuperscript{th} output(input) constraint in \eqref{Eq:Nominal} or \eqref{Eq:RobustBoxDEA}. We denote the values that $\sm,~\sn$ take at an optimal solution to \eqref{Eq:Nominal} by 
$\smm,~\snn$ and at an optimal solution to \eqref{Eq:RobustBoxDEA} by $\smu,~\snu$. In this way, any binding constraints in \eqref{Eq:Nominal} or \eqref{Eq:RobustBoxDEA} will have slack variables equal to zero.
\begin{proposition} \label{Prop:BindingInput}
For inefficient DMU $\ih$ and for any optimal solution $(\Lambda, E^{\ih})$ to \eqref{Eq:Nominal} there exists at least one binding input constraint.
\end{proposition}
\begin{proof} Because \eqref{Eq:Nominal} is a linear program, at an optimal solution there is always at least one binding constraint. Assume that this is an output constraint $p$ and that there are no binding input constraints. Then at an optimal solution, each input constraint, $n \in \mathcal{N}$, can be written as: 
\begin{equation}
\frac{\NX_n\Lll }{\NX_{n\ih}}+\frac{\snn}{\NX_{n\ih}}=E^{\ih}.
\end{equation}
Let input $q$ have the largest value of $\frac{\NX_n \Lll }{\NX_{n\ih}}$. Given that $\NX_{ni}>0~ \forall~ n,i$,
\begin{equation}
\label{Prop:BindingInputConst}
  \frac{ \NX_{q}\Lll}{\NX_{q\ih}} \geq  \frac{\NX_{n}\Lll}{\NX_{n\ih}}~\Leftrightarrow~\frac{\sqq}{\NX_{q\ih}}\leq \frac{\snn}{\NX_{n\ih}}~\forall~n ~\neq q. 
\end{equation}
Reducing $\sqq$ to 0 reduces $E^{\ih}$ by $\frac{\sqq}{X_{q\ih}}$. Then all remaining $\frac{\snn}{\NX_{n\ih}}$ must reduce by $\frac{\sqq}{\NX_{q\ih}}$. Because \eqref{Eq:Nominal} is a minimisation problem, $\sqq=0$ is feasible, which contradicts the assumption that there are no binding input constraints. 
\end{proof}
Similarly, Proposition \ref{Prop:BindingInput} holds for any given uncertainty in \eqref{Eq:Robust}. We show that for box uncertainty, Definition \ref{Def:BoxUncertainty}, there always exists an optimal solution to \eqref{Eq:RobustBoxDEA}. 

\begin{theorem}\label{Thm:ExtremePoints}
For box uncertainty, Definition \ref{Def:BoxUncertainty}, there exists an optimal solution, $\left( \hat{\Lambda}, \mathcal{E}^{\ih}(\mathcal{U})\right),$ to \eqref{Eq:RobustBoxDEA} such that $\Dym^* \in \{-\un,\un\},~\Dxn^* \in \{-\un,\un\},~m\in \mathcal{M},~ n\in \mathcal{N},~i \in \mathcal{I}.$
\end{theorem}
\begin{proof}
Since the nominal problem \eqref{Eq:Nominal}, is the  special case of \eqref{Eq:RobustBoxDEA} when we have no uncertainty, we inherit the feasibility of $(\lambda,\Ti)=(0,0, \dots, 1,\allowbreak \dots ,0,1)$ in \eqref{Eq:RobustBoxDEA}. 
When uncertainty is introduced the nominal data $\NX,\NY$ can change and as a result the PPS may change. The extreme points of the `new' PPS will be extreme points from the nominal data PPS or extreme points of the uncertainty sets.
An optimal solution of a linear function over a convex polyhedron will exist at an extreme point of the feasible region. 
\end{proof}
Furthermore, for box uncertainty, when $\Omega$ is unrestricted we show that all DMUs are capable.
\begin{theorem} \label{Thm:NeverIncapable}
For box uncertainty, Definition \ref{Def:BoxUncertainty}, when $\nu=\infty$, there can never be an incapable DMU.
\end{theorem}
\begin{proof}
If $E^i=1$, DMU $i$ is efficient and hence, capable. For any inefficient DMU $\ih$, $E^{\ih}<1$. Choose an input $n'\in \mathcal{N}$ and let $\nabla_{n'\ih}=X_{n'\ih}-\epsilon,~\nabla_{ni}=0,~\Delta_{mi}=0 ~\forall~ i \in \mathcal{I},~m \in \mathcal{M}$ and $n \in \mathcal{N}$. Here, we can choose any input, $n'\in \mathcal{N}$, as all other data remain fixed to their nominal values. Then $\map(\U)=X_{n'\ih}-\epsilon$ and $\bar{X}_{n'\ih}=X_{n'\ih}-\nabla_{n'\ih}=\epsilon< min_{i\neq\hat{i}}  \bar{X}_{n’ i}$ is a possible realisation of the uncertain data and 
Then $(\hat{\Lambda}_i=0, ~\forall~i \in \mathcal{I},~ i\neq \ih,\hat{\Lambda}_{\ih}=1,~\Eih=1)$ is an optimal solution to \eqref{Eq:RobustBoxDEA} and DMU $\ih$ is capable. 
\end{proof}
In practice, however, it is more likely that $\map(\U)$ will be restricted and there will be an upper bound on the amount of uncertainty. This value should be problem specific. For example, in Section \ref{SubSec:CaseStudy}, the upper bound on the amount of uncertainty is selected to reflect international guidelines on the delivered dose of radiation.

\citet{ehrgott2018uncertain} prove that as the amount of uncertainty increases the efficiency score will increase. Here, we require more specific results for the coming sections. In Lemmas \ref{Lemma:SumIncrease} and \ref{Lemma:SumDecrease},we show that an increase(decrease) in the weighted sum of the outputs can cause an increase(decrease) in the weighted sum of the inputs and hence affect $\gamma^*_{\ih}$ despite $\theta^{\ih}$ not appearing directly in the output constraints in \eqref{Eq:Nominal}.
We define $Q$ to be the set of binding input constraints in \eqref{Eq:Nominal}. First, we consider an increase in the weighted sum of the uncertain outputs $\UY$ compared to the nominal outputs $\NY$.
\begin{lemma}\label{Lemma:SumIncrease}
If $\NY_m \Lll < \UY_m \Llu$, then $\NX_q\lii \leq \UX_q \liu,~q\in Q$ and hence $\gamma^*_{\ih} \geq \Tih$. 
\end{lemma}
\begin{proof} 
Consider $\NY_{mk}<\UY_{mk},~k \in \mathcal{I},~ k \neq \ih$. Let an optimal solution to \eqref{Eq:uDEA} be $\left(\Lambda^*,\gamma^*_{\ih},\right.$ $\left. \Delta_{mi}^*=0,~\Delta_{mk}^*=\un, i \neq k \right) $. At an optimal solution to \eqref{Eq:uDEA}, the LHS of the $m\textsuperscript{th}$ constraint decreases by $\Lambda_{k}^* \un$ compared to the $m\textsuperscript{th}$ constraint in \eqref{Eq:Nominal}. If $\Lambda_{k}^* \un \leq \smm$, then the slack in the $m\textsuperscript{th}$ constraint at a uDEA optimal solution decreases by $\Lambda_{k}^* \un$ compared to the $m\textsuperscript{th}$ constraint in \eqref{Eq:Nominal}. Then $\smu=\smm-\Lambda_{k}^* \un$, and the decrease in output $\NY_{mk}$ has no effect on the efficiency score.
If $\Lambda_{k}^* \un>\smm$, then $\smu=0$ and $ \NY_m\Lll$ increases such that $\UY_m\Llu =\NY_m\Lll+\Lambda_{k}^* \un - \smm$, therefore $\UY_m\Llu >\NY_m\Lll$. Constraint $q$ is binding, therefore any change in $\Lambda$ will cause the weighted sum of the inputs to increase (or stay the same). 
\end{proof}

Similarly, a decrease in the weighted sum of the outputs results in the weighted sum of the inputs decreasing (or staying the same). 
\begin{lemma}\label{Lemma:SumDecrease}
If $\NY_m\Lll>\UY_m\Llu$ then $\NX_q \lii\geq \UX_q \liu,~q\in Q$ and hence $\gamma^*_{\ih} \leq \Ti$.
\end{lemma}
\begin{proof}
Analogous to Lemma \ref{Lemma:SumIncrease}.
\end{proof}
Similarly, an increase in the PPS results in the efficiency score of DMU $\ih$ staying the same or reducing.
\begin{proposition}\label{Prop:IncreasePPS}
Let $T,~T'$ be PPSs of \eqref{Eq:Nominal} such that $T'\subset T$. Let $\Ti(\theta^{i'})$ be the efficiency score for DMU $i$ in $T(T')$ then $\Ti \leq \theta^{i'}$. 
\end{proposition}
\begin{proof}
Let $\Tih<1(\theta^{\ih'}<1)$ be the efficiency score for DMU $\ih$ in the nominal DEA problem \eqref{Eq:Nominal} with PPS $T(T')$. DMU $\ih$ is inefficient and has nominal DEA projection points $(\Tih x^{\ih},y^{\ih}) \in T$, $(\theta^{\ih'} x^{\ih},y^{\ih}) \in T'$. $T\subset T'$ therefore, $\Tih x^{\ih}\geq \theta^{\ih'}x^{\ih}$ and $\Tih \geq \theta^{\ih'}$ because $x^{\ih}$ is fixed. 
\end{proof}
Consequently, an increase in the PPS results in the efficiency score of DMU $\ih$ staying the same or reducing.

We now consider the possible optimal solutions for DMU $\ih$. We show that there exists an optimal solution to \eqref{Eq:uDEA} for box uncertainty, Definition \ref{Def:BoxUncertainty}, for DMU $\ih$ such that DMUs that are inefficient in \eqref{Eq:Nominal} will not become peers to DMU $\ih$. In this way, when investigating the effect of uncertainty on DMU $\ih$ we only consider DMUs that are efficient in the nominal DEA problem \eqref{Eq:Nominal}.
\begin{proposition}\label{Prop:NoPeers}
Let DMU $l,~l\in\mathcal{I},~l\neq \ih$, be an inefficient DMU, $E^l<1$. For DMU $\ih$ with $E^{\ih}\leq 1$ and box uncertainty, Definition \ref{Def:BoxUncertainty}, there exists an optimal solution to \eqref{Eq:uDEA},  $(\Lambda^*,\gamma^*_{\ih},\Delta^*,\nabla^*)$, such that $\Lambda^*_l=0$. 
\end{proposition}
\begin{proof}
For any fixed amount of uncertainty, \eqref{Eq:uDEA} becomes \eqref{Eq:RobustBoxDEA}. The PPS of \eqref{Eq:RobustBoxDEA} is dependent on $\map(\U)$, the amount of uncertainty. From Theorem \ref{Thm:ExtremePoints}, an optimal solution to \eqref{Eq:RobustBoxDEA} will occur at an extreme point of the feasible region. Let $T$ be the PPS of the nominal DEA problem \eqref{Eq:Nominal} and $T'$ be the smallest possible PPS of \eqref{Eq:RobustBoxDEA} when we have $\map(\U)$ such that $\Delta=\Delta^*$ and $ \nabla=\nabla^*$. Then, $T'\subset T$, and from Proposition \ref{Prop:IncreasePPS}, $\gamma^*_{\ih}=\Eih\geq E^{\ih}$.
If $\Lambda^*_l>0, ~\mathcal{E}^{l}(\mathcal{U})=1$ and DMU $l^u$ must lie on the efficient frontier at the boundary of $T'$. DMU $l^u$ cannot be an extreme point of $T'$ as it is not an extreme point of $T$. Therefore, if $\mathcal{E}^{l}(\mathcal{U})=1$,  DMU $l^u$ lies on the interior of an efficient facet of $T'$ and can be written as a convex combination of DMUs that are efficient in \eqref{Eq:Nominal}. Hence, there is an optimal solution such that $\Lambda^*_{l}=0$ and since $\Delta=\Delta^*$ and $ \nabla=\nabla^*$, $\Lambda^*_{l}=0$ at an optimal solution to \eqref{Eq:uDEA}.
\end{proof}
In this way, DMUs that are efficient in the nominal DEA problem \eqref{Eq:Nominal} remain efficient in the uDEA model \eqref{Eq:uDEA}. Therefore, we assume that from the $I$ DMUs, we have $I-1$ efficient DMUs and a single inefficient DMU, DMU $\ih$. Then $E^{\ih}<1$ and $\Tii=1~ \forall i\in \mathcal{I},~i\neq \ih$. DMUs that are efficient in \eqref{Eq:Nominal} will remain efficient, at least until the amount of uncertainty is equal to the minimum amount of uncertainty required for DMU $\ih$ to become efficient. 


\section{The specific case of box uncertainty}\label{Sec:BoxUncertainty}
We now focus on the specific case of box uncertainty. Therefore, when we refer to \eqref{Eq:uDEA}, we assume we have box uncertainty as defined in Definition \ref{Def:BoxUncertainty}. In Section \ref{Sec:FixedU}, we consider a fixed amount of uncertainty, i.e. the robust DEA model \eqref{Eq:RobustBoxDEA}, then in Section \ref{SubSec:DEADistance}, we introduce the concept of DEA distance which allows us to explore the effect of changing uncertainty in Section \ref{SubSec:ChangeU}.
\subsection{Fixed uncertainty}\label{Sec:FixedU}
First, we consider fixed uncertainty, i.e. we omit $\min_{\U \in \Omega}$ in \eqref{Eq:uDEA} and so we consider \eqref{Eq:RobustBoxDEA}. 
By considering single and multiple changes in the data of the inefficient or efficient DMUs and then multiple changes to all DMUs, we show that for a fixed amount of uncertainty $\un$, the following theorem holds.
\begin{theorem} \label{Thm:MaxIncUFix}
To maximise the possible increase in efficiency score for DMU $\ih$, solving \eqref{Eq:RobustBoxDEA} will result in the following uncertainty being selected:\\ $\Delta_{mi}=\Dxk=-\un,~\Dyk=\nabla_{ni}=\un,~i \in \mathcal{I},~i \neq\ih,~~m\in \mathcal{M},~n \in \mathcal{N}$.
\end{theorem}
Consider a single change in the data, i.e. a single $\NX_{ni}$ or $\NY_{mi}$ changes by $\pm \un$ while the remaining data are fixed. Binding input constraints result in the largest change in efficiency score. This can be seen in Lemma \ref{Lemma:SumIncrease}, where the role of slack variables and the possible changes in the weighted sum of the inputs and outputs are demonstrated. Hence, we consider only binding input constraints. By considering the possible single changes, we have the following lemmas.
\begin{lemma}\label{Lemma:EffDMUInput}
For a single input $q\in \mathcal{N}$ of a DMU changed by $\un$, the maximum increase in $E^{\ih}$ that can occur 
is $\frac{\un}{\NX_{q\ih}}$. 
\end{lemma}
\begin{proof}
Consider an increase in an efficient DMU $r$'s input, $q$. The LHS of the $q$\textsuperscript{th} constraint in \eqref{Eq:RobustBoxDEA} increases by $\lru \un$ compared to the $q$\textsuperscript{th} constraint in \eqref{Eq:Nominal}. Therefore, $\NX_q\lii$ must decrease or $E^{\ih} \NX_{nq}$ must increase. However, $\NX_q \lii$ cannot decrease because the constraint for input $q$ is binding so $\NX_q \lii$ is already as small as possible. This means $\NX_q \liu=\NX_q \lii$. DMU $\ih$'s data are fixed, therefore $E^{\ih}$ must increase by $\gamma^*_{\ih}-E^{\ih}=\frac{\un\lru}{\NX_{q\ih}}$. This is maximal when $\lru=1$ and $E^{\ih}$ increases by $\frac{ \un}{\NX_{q\ih}}$. Similarly, if a single input of DMU $r$ decreases by $\un$, $E^{\ih}$ must stay the same or decrease by a maximum of $\frac{\un }{\NX_{q\ih}}$. 

Likewise, for a single input $q$ of DMU $\ih$ changed by $\un$, the maximum increase in $E^{\ih}$ that can occur is $\frac{\un}{\NX_{q\ih}}$. This occurs when the input of DMU $\ih$ is reduced by $\un$. Therefore, the maximum increase in $E^{\ih}$ that can occur when a single input of a DMU is changed by $\un$ is $\frac{\un}{\NX_{q\ih}}$.
\end{proof}
\begin{lemma}\label{Lemma:EffDMUOutput}
For a single output of an efficient DMU changed by $\un$, the maximum increase in $E^{\ih}$ that can occur is 
 \begin{equation}\label{Eq:EffDMUOutput}
   \max _{~q \in Q}~ \frac { \max _{\substack{i\neq \ih}} \NX_{qi}- \min _{\substack{i\neq \ih}} \NX_{qi} }{\NX_{q\ih}}.
 \end{equation}
\end{lemma}
\begin{proof}
This follows from Lemmas \ref{Lemma:SumIncrease} and \ref{Lemma:SumDecrease}. 
\end{proof}
\begin{lemma}\label{Lemma:IneffOut}
For a single output of DMU $\ih$ changed by $\un$, if $\Dyk=\un,~ \gamma^*_{\ih} \geq \theta^{\ih},$ otherwise if $\Dyk=-\un, ~\gamma^*_{\ih} \leq \theta^{\ih}$. 
\end{lemma}
\begin{proof}
Analogous to Lemma \ref{Lemma:SumIncrease}. 
\end{proof}
Therefore, from Lemmas \ref{Lemma:EffDMUInput}, \ref{Lemma:EffDMUOutput} and \ref{Lemma:IneffOut}, we conclude the following result.
\begin{result}\label{Result:Changes}
DMU $\ih$'s efficiency score will not decrease if any of the following occur: DMU $\ih$'s inputs decrease or outputs increase or an efficient DMU $r$'s inputs increase or outputs decrease.
\end{result}
The analysis of Lemmas \ref{Lemma:EffDMUInput}-\ref{Lemma:EffDMUOutput} can be repeated for multiple changes in inputs and outputs for the efficient and inefficient DMUs. Multiple changes in the data result in a maximum increase in DMU $\ih$'s efficiency score which is not smaller than that caused by a single change in the data. Consequently, we consider the changes in efficiency score if all the data of all DMUs change. From Result \ref{Result:Changes}, we only want to consider scenarios that increase DMU $\ih$'s efficiency score. Therefore, we only need to consider an increase(decrease) in efficient(inefficient) DMUs' inputs and a decrease(increase) in efficient(inefficient) DMUs' outputs.
\begin{proposition}\label{Prop:AllChange}
 The maximum increase in efficiency score for DMU $\ih$ when changing the data of all DMUs is 
 \begin{equation}
  { \max _{~q \in Q}~ \frac { \max _{\substack{i\neq \ih}} \NX_{qi}- \min _{\substack{i\neq \ih}} \NX_{qi}+2\un }{\NX_{q\ih}}} .
 \end{equation}
\end{proposition}
\begin{proof}
From Result \ref{Result:Changes}, changing the outputs by $\Dym=-\un,~\Dyk=\un,~i \neq \ih$, results in the efficiency score increasing or remaining the same. When the inputs change, the largest increase in efficiency score occurs when $\Dxn=\un,~\Dxk=-\un,~i \neq \ih$. Choose a binding input constraint, $q$. At an optimal solution to \eqref{Eq:RobustBoxDEA} for DMU $\ih$, the introduction of uncertainty means the LHS of the $q$\textsuperscript{th} input constraint increases by $\un(1+\Eih)$ compared to the $q$\textsuperscript{th} input constraint in \eqref{Eq:Nominal}, i.e.
\begin{subequations}
\label{Eq:ChangeBindALL}
\addtocounter{parentequation}{-1}
\begin{align}
 (\NX_{q}+\un )\bar{\Lambda}- (\NX_{q\ih}-\un)\Ei = \NX_{q}\lii-\NX_{q\ih} E^{\ih}  \label{Eq:ChangeBind1} \\
 \Eih-E^{\ih}=\frac { \NX_q \bar{\Lambda}- \NX_q\lii +\un(1 +\Eih) }{\NX_{q\ih}}.\label{Eq:ChangeBind2}
 \end{align} 
\end{subequations}
The maximum change in efficiency score occurs when $ \NX_q \bar{\Lambda}$ and $\Eih$ are as large as possible and $ \NX_{q}\lii $ is as small as possible, and this gives our result.
\end{proof}
This completes the proof of Theorem \ref{Thm:MaxIncUFix} and gives the following result.
\begin{result}\label{Result:OppChanges}
There is always an optimal solution to \eqref{Eq:RobustBoxDEA} such that $\Eih \geq E^{\ih}$ in which DMU $\ih$ benefits from increased outputs and decreased inputs and the efficient DMUs' inputs increase and outputs decrease. 
\end{result}

\subsection{DEA distance}\label{SubSec:DEADistance}
Solving \eqref{Eq:uDEA} determines the smallest amount of uncertainty required for DMU $\ih$ to be deemed efficient. To aid solving \eqref{Eq:uDEA}, we define the DEA distance for DMU $\ih$ to measure how far DMU $\ih$ is from the efficient frontier.

We are considering an input oriented DEA model and hence, the inefficient DMUs are projected to the efficient frontier while their outputs remain fixed. 
\begin{definition}\label{Def:DEADist}
 The DEA distance from DMU $\ih$ to its target point on the hyperplane, $\NM \in \Psi $ is 
 \begin{equation}\label{Eq:DEADistance}
  \mathfrak{D}(\ih,\NM) = \frac{|\Ax \NX_{1\ih}+\dots + \Axn \NX_{N\ih}+\By \NY_{1\ih}+\dots +\Bym \NY_{M\ih}-d|}{\sqrt{\Ax ^2+\dots +\Axn ^2}}.
 \end{equation}
\end{definition}
Definition \ref{Def:DEADist} can be derived in the following manner. DMU $\ih$ is inefficient and hence, does not lie on the hyperplane $\NM \in \Psi$. Consider the $N-$dimensional hyperplane $\Nn$, where the outputs are fixed to be the value they take for DMU $\ih$. 
\begin{equation}\label{Eq:PlaneN-1}
  \Nn= \{r \in \RR^N:r^T \eta_2=d-\By \NY_{1\ih}-\dots -\Bym \NY_{M\ih}\},
\end{equation}
where $\eta_2=(\Ax, \dots, \Axn)^T$. Then the DEA distance from DMU $\ih$ to it's projection on $\NM$ is found by calculating the Euclidean distance from DMU $\ih$ to the plane $\Nn$. 
\begin{proposition}
 \label{Prop:DEADist}
$\mathfrak{D}(i,\NM),$ is the Euclidean distance from DMU $i$ to $T(i,\NM)$ given by \eqref{Eq:DEADistance}.
\end{proposition}
\begin{definition}\label{Def:DEADistmin}
 \[\mathfrak{D}(i):=\min_{\NM \in \Psi}\mathfrak{D}(i,\NM).\]  
 \end{definition}
$\mathfrak{D}(i)$ is the minimum DEA distance from DMU $i$ to all hyperplanes $\NM \in \Psi$. It follows from Definition \ref{Def:DEADistmin} that DMU $i$ is efficient if and only if $\mathfrak{D}(i)=0$. In the nominal DEA problem \eqref{Eq:Nominal}, the hyperplane $\psi$ that DMU $i$ is projected to, will always give the minimum DEA distance $\mathfrak{D}(i)$.
\begin{definition}
 $\minU$ is the minimum amount of uncertainty required for DMU $\ih$ to be deemed efficient.
\end{definition}
From Definitions \ref{Def:VirtualDMU}-\ref{Def:DEADistmin} we can compute $\minU$, the minimum amount of uncertainty such that $\mathfrak{D}(\ih^u)=0$. We note that for any DMU $i$ that is efficient in the nominal DEA problem $\minU=0$.
\subsection{Changing uncertainty}\label{SubSec:ChangeU}
We now consider the case where we have a changing value of $\un$ to determine the minimum amount of uncertainty, $\minU$, required such that DMU $\ih$ becomes efficient. \citet{ehrgott2018uncertain} prove that as the amount of uncertainty increases the efficiency score will increase. In the case of box uncertainty this is intuitive, our collection of uncertainty are boxes, so an increase in $\map(\U)$ means the box size $\sigma$ increases.
\begin{corollary} \label{Cor:InEff}
For DMU $\ih$, with box uncertainty, increasing the value of $\un$ increases the efficiency score. 
\end{corollary}
\begin{proof}
Consider an increase in $\un$ such that $\bar{\un}=\un+\epsilon$. Let $\left( \hat{\Lambda}, \mathcal{E}^{\ih}(\mathcal{U})\right)$ be an optimal solution to \eqref{Eq:RobustBoxDEA} when $\map(\U)=\un$ and $\left( \bar{\Lambda},\bar{ \mathcal{E}}^{\ih}(\mathcal{U})\right)$ be an optimal solution when $\map(\U)=\bar{\un}$. When $\map(\U)=\bar{\un}$, select a binding input constraint $q$ in \eqref{Eq:RobustBoxDEA}. At an optimal solution $\left( \bar{\Lambda},\bar{ \mathcal{E}}^{\ih}(\mathcal{U})\right),$ the LHS of the $q$\textsuperscript{th} input constraint increases by $\epsilon(1+\bar{ \mathcal{E}}^{\ih}(\mathcal{U}))$ compared to the corresponding constraint in \eqref{Eq:RobustBoxDEA} with $\map(\U)=\un$. Consequently, $ (\NX_{q}+\un)\hat{\Lambda}>(\NX_{q}+\un) \bar{\Lambda} $ or $\mathcal{E}^{\ih}(\mathcal{U})( \NX_{q\ih}-\un)<\bar{ \mathcal{E}}^{\ih}(\mathcal{U})( \NX_{q\ih}-\un)$. However, $ (\NX_{q}+\un)\hat{\Lambda}$ cannot decrease because the constraint for input $q$ is binding. $(\NX_{q}+\un)\hat{\Lambda}$ is already as small as possible and $ (\NX_{q}+\un)\hat{\Lambda}=(\NX_{q}+\un) \bar{\Lambda} $. Then the introduction of $ \bar{\un}=\un+\epsilon$ results in
\begin{equation}\label{EQ}
    \bar{ \mathcal{E}}^{\ih}(\mathcal{U}) -\mathcal{E}^{\ih}(\mathcal{U}) =\frac{\epsilon(1+\bar{ \mathcal{E}}^{\ih}(\mathcal{U}))}{\NX_{q\ih}-\un}.
\end{equation}
The RHS of \eqref{EQ} is greater than zero therefore, $\bar{ \mathcal{E}}^{\ih}(\mathcal{U}) -\mathcal{E}^{\ih}(\mathcal{U})>0$ and the efficiency score is increasing. 
\end{proof}

In this way, $E^{\ih}$, the efficiency score when there is no uncertainty, provides a lower bound for DMU $\ih$'s efficiency score when uncertainty is present.
\begin{theorem}\label{Thm:ValueOfU}
For DMU $\ih$, the minimum amount of uncertainty required such that uncertain DMU $\ih^u$ is at the target point $T(\ih^u,\NMu)$ on $\NMu$ is
\begin{equation}\label{Eq:ValueOfUALL}
    \un=\frac{|\Ax \NX_{1\ih}+\dots + \Axn \NX_{N\ih}+\By \NY_{1\ih}+\dots +\Bym \NY_{M\ih}-d|}{2|-\Ax-\dots -\Axn+ \By+ \dots +\Bym|} . 
\end{equation}
\end{theorem}
\begin{proof}
The DEA distance from the point $\ih$ to the hyperplane $\NM$ is given by \eqref{Eq:DEADistance}. When uncertainty is introduced, we consider the DEA distance from the virtual DMU $\ih^u$ to a plane $\NMu$. $\NMu$ is parallel to $\NM$ and given by the equation
\begin{equation}\label{Eq:PlaneNMu}
  \NMu=\left\{r\in \RR^\Phi :r^T \eta_1=d'\right \},
\end{equation}
where $d'=d +\un(\Ax+\Axx + \cdots +\Axn -\By -\Byy - \cdots - \Bym)$. Consider the $N-$dimensional hyperplane $\Nnu$, where the outputs are fixed to be the value they take for virtual DMU $\ih^u$,
\begin{equation}\label{Eq:PlaneN-1u}
  \Nnu =\left\{r\in\RR^N:r^T \eta_2=d'-\By \NY_{1\ih}-\dots -\Bym \NY_{M\ih}\right\}.
\end{equation}
The DEA distance from DMU $\ih^u$ to the hyperplane $\NMu$ is found by projecting DMU $\ih^u$ to $\NMu$ along $\Nnu$, 
\begin{equation}
\begin{aligned}
\mathfrak{D}(\ih^u,\NMu)&=\frac{|\Ax (\NX_{1\ih}-\un)+\dots + \Axn (\NX_{N\ih}-\un)+\By (\NY_{1\ih}+\un) +\dots +\Bym (\NY_{M\ih}+\un)-d|}{\sqrt{\Ax ^2+\dots +\Axn ^2}}  \\
&=\mathfrak{D}(\ih,\NM)+\frac{|2\un (-\Ax-\dots -\Axn+ \By+ \dots +\Bym)|}{\sqrt{\Ax ^2+\dots +\Axn ^2}}.\label{Eq:d(k,n-1)u} 
\end{aligned}
\end{equation}
Therefore, by introducing an uncertainty of $\un$, the DEA distance reduces by at most $\frac{|2\un (-\Ax-\dots -\Axn+ \By+ \dots +\Bym)|}{\sqrt{\Ax ^2+\dots +\Axn ^2}}$.
For DMU $\ih^u$ to be on the facet of the efficient frontier defined by $\NMu$, we require uncertainty such that the DEA distance from the point $\ih^u$ to the hyperplane $\NMu$ is 0, i.e. $\mathfrak{D}(\ih^u,\NMu)=0$. Substituting $\mathfrak{D}(\ih^u,\NMu)=0$ into \eqref{Eq:d(k,n-1)u} and rearranging to make $\un$ the subject gives \eqref{Eq:ValueOfUALL}. 
\end{proof}
Theorem \ref{Thm:ValueOfU} allows us to calculate the amount of uncertainty required for virtual DMU $\ih^u$ to be deemed efficient on $\NMu \in \Psi^u$. This is done by projecting virtual DMU $\ih^u$ to a point $T(\ih^u,\NMu)$ on $\NMu \in \Psi^u$. From Theorem \ref{Thm:ValueOfU}, we can solve \eqref{Eq:uDEA} in the following way. First compute the minimum amount of uncertainty for DMU $\ih^u$ to be projected to $T(\ih^u,\NMu)$ for all $\NMu \in \Psi^u$, \eqref{Eq:ValueOfUALL}, then select the minimum of these, $\mathfrak{D}(\ih^u)$. This gives $\minU$, the minimum amount of uncertainty for DMU $\ih$. 

Finding all $\NM \in \Psi$ for large $M,~N$ and/or $I$ can be computationally intensive. However, when $N=M=1$, \eqref{Eq:ValueOfUALL} from Theorem \ref{Thm:ValueOfU} can be simplified and hence $\upsilon_{i}^*$ can be calculated easily. Consider a DEA problem \eqref{Eq:Nominal} with three DMUs $A$, $B$, $C$ and $M=N=1$. Let DMU $C$ be inefficient with peers $A$ and $B$ and let $g=\frac{y^{B}-y^{A}}{x^{B}-x^{A}}$. Then, from Theorem \ref{Thm:ValueOfU} with $M=N=1$, the minimum amount of uncertainty required to be projected to the target point $T(C^u,\Pi^u_{AB})$ on $\Pi_{AB}^u$ defined by the inputs and outputs of DMU $A$ and $B$ is 
\begin{equation} \label{Eq:ValueOfU2D}
\un=\frac{ g(x^{C}-x^{A})-y^{C}+y^{A}}{2(1+ g)}.
\end{equation}
Once the minimum amount of uncertainty required for DMU $\ih^u$ to be projected to all $\Pi ^u \in \Psi^u$ has been calculated, $\minU$ can be found by selecting the minimum of these. We demonstrate this for $M=N=1$ in Example \ref{Ex:DEADist2D}.

\begin{example}{\textbf{DEA Distance.}}\label{Ex:DEADist2D}
Consider the six DMUs pictured in Figure \ref{Fig:DEADist2D} whose nominal data are listed in Table \ref{Tab:2DExample}. In Figure \ref{Fig:DEADist2D}, the efficient frontier is shown by the red line sections between the efficient DMUs $A-D$. By extending the line sections going through pairs of efficient DMUs' data (and the horizontal and vertical lines through the DMUs with the largest(smallest) $\Ny(\Nx)$ value), we can show the hyperplanes, (here lines) that define the PPS. The efficient frontier has five lines that intersect with the PPS as shown in Figure \ref{Fig:DEADist2D}. This means $\Psi=\{ \Pi_A,\Pi_{AB},\Pi_{BC},\Pi_{CD},\Pi_{D}\}$. Where $\Pi_{ij}$ is the line going through the points defined by DMU $i,j\in\{A,B,C,D\}$, $\Pi_A$ is the vertical line $x=x^{A}$ and $\Pi_D$ is the horizontal line $y=y^{D}$. Here we include $\Pi_D \in \Psi$ even though it only intersects the efficient frontier at the point $D$. This will be discussed further when we calculate the minimum DEA distances in Table \ref{Tab:DEADist2DUncertain}.
\begin{table}[bt]
\centering
\begin{tabular}{llllll}
\hline
\textbf{DMU} &\textbf{In} & \textbf{Out} & $\mathbf{E^{\ih}}$&\textbf{Uncertain In} & \textbf{Uncertain Out} \\ \hline
A & 1 & 1 & 1 & 1+\un & 1-\un \\ 
B & 3 & 4 & 1 &3+\un & 4-\un\\ 
C & 7 & 7 & 1 & 7+\un & 7-\un \\ 
D & 10 & 8 & 1& 10+\un & 8-\un \\ 
E & 8 & 5 & 0.542 & 8-\un & 5+\un  \\ 
F & 6 & 2 & 0.278& 6-\un & 2+\un \\ \hline
\end{tabular}
\caption{Nominal and uncertain data, Example \ref{Ex:DEADist2D}.}
\label{Tab:2DExample}
\end{table}
\begin{figure}[bt]
\centering
\includegraphics[scale=0.8]{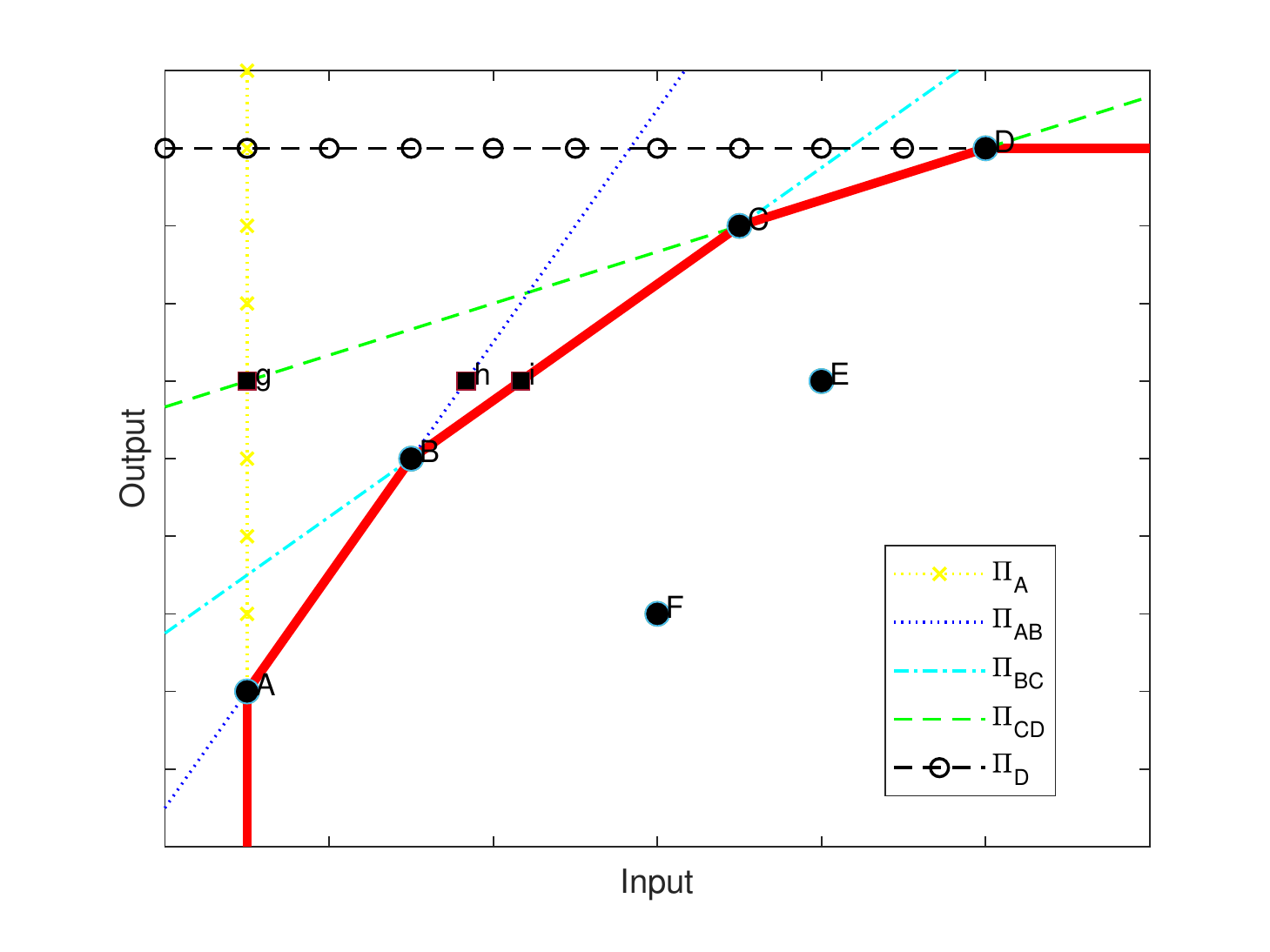}
\caption{The hyperplanes that define the PPS, Example \ref{Ex:DEADist2D}.}\label{Fig:DEADist2D}
\end{figure}
The DEA distances for the inefficient DMUs to their target points on each line in $\Psi$ are given in Table \ref{Tab:DEADist2D}. The target points for DMU $E$ for each of the lines are shown in Figure \ref{Fig:DEADist2D} by the points $g, h$ and $i$. We note here that in the nominal DEA DMUs $E$ and $F$ can never be projected to the line $\Pi_D$.
\begin{table}[tb]
\centering
\begin{tabular}{lllll}
\hline
$\mathbf{\Pi \in \Psi}$ & $\mathbf{T(E,\Pi)}$ & $\mathbf{T(F,\Pi)}$ & \textbf{$\mathbf{\mathfrak{D}(E,\Pi)}$} & \textbf{$\mathbf{\mathfrak{D}(F,\Pi)}$} \\ \hline
$\mathbf{\Pi_A}$ & $g=$(1,5) & (1,2) & 7 & 5 \\ 
$\mathbf{\Pi_{AB}}$ & $h=(\frac{11}{3},5)$ & $(\frac{5}{3},2)$ & $\frac{13}{3}$ & $\frac{13}{3}$ \\ 
$\mathbf{\Pi_{BC}}$ & $i=(\frac{1}{3},5)$ & $(\frac{1}{3},2)$ & $\frac{11}{3}$ & $\frac{17}{3}$ \\ 
$\mathbf{\Pi_{CD}}$ & $g=$(1,5) & (-8,2) & 7 & 14 \\ 
$\mathbf{\Pi_{D}}$ & \multicolumn{1}{c}{-} & \multicolumn{1}{c}{-} & $\infty$ & $\infty$ \\ \hline
\end{tabular}
\caption{DEA Distance for DMUs $E$ and $F$, Example \ref{Ex:DEADist2D}.}
\label{Tab:DEADist2D}
\end{table}
Here $\mathfrak{D}(E)=3.66$ is obtained when DMU $E$ is projected to $\Pi_{BC}$ as DMU $E$ has peers DMU $B$ and $C$ in the nominal DEA. Similarly, $\mathfrak{D}(F)=4.33$ is obtained when DMU $F$ is projected to $\Pi_{AB}$ as DMU $F$ has peers DMU $A$ and $B$ in the nominal DEA. 

\begin{table}[htb]
\centering
\begin{tabular}{lllll}
\hline
$\mathbf{\NMu \in \Psi^u}$ & $\mathbf{\mathfrak{D}(E^u,\NMu)}$ & $\mathbf{\mathfrak{D}(F^u,\NMu)}$ & $\mathbf{\un_E}$ & $\mathbf{\un_F}$ \\ \hline
$\mathbf{\Pi_A^u}$ & 7-2\un & 5-2\un & 3.50 & 2.50 \\ 
$\mathbf{\Pi_{AB}^u}$ & $| \frac{13}{3}-\frac{10}{3}\un |$ & $|\frac{13}{3}-\frac{10}{3}\un |$ & 1.30 & 1.30 \\ 
$\mathbf{\Pi_{BC}^u}$ & $|\frac{11}{3}-\frac{14}{3}5\un |$ & $|\frac{17}{3}-\frac{14}{3}\un |$ & 0.79 & 1.21 \\ 
$\mathbf{\Pi_{CD}^u}$ & $|7-8\un |$& $|14-8\un |$& 0.88 & 1.75 \\ 
$\mathbf{\Pi_{D}^u}$ & \multicolumn{1}{c}{-} & \multicolumn{1}{c}{-}& $\un > 1.50$ & $\un > 3$ \\ \hline
\end{tabular}
\caption{DEA Distance and minimum amount of uncertainty for DMUs $E$ and $F$, Example \ref{Ex:DEADist2D}.}
\label{Tab:DEADist2DUncertain}
\end{table}
To calculate $\mathfrak{D}(\ih^u,\NMu),$ for inefficient DMUs $E^u$ and $F^u$, we transform the nominal data according to Definition \ref{Def:VirtualDMU}. This gives the uncertain data in Table \ref{Tab:2DExample}. The DEA distance of the transformed data can then be calculated, see Table \ref{Tab:DEADist2DUncertain}. As a result, we calculate the amount of $\un$ for each $\NMu \in \Psi^u$. To find the minimum amount of uncertainty required for DMUs $E$ and $F$ to become efficient we find $\un^*=\min_{\NMu \in \Psi^u} \{\un~s.t.~\mathfrak{D}(\ih,\NMu)=0\}$. 
These are shown in Table \ref{Tab:DEADist2DUncertain}. We note that we do not have DEA distances to $\Pi_D^u$ but we can still calculate a minimum amount of uncertainty. This is because we consider an input oriented DEA model. To calculate the amount of uncertainty that corresponds to this vertical projection for $\Pi_D^u$, define a virtual DMU $d^u$ such that $x^{d^u}=x^{D^u}+\kappa,~y^{d^u}=y^{D^u} +\epsilon$, where $\kappa$ is a positive constant such that $x^{D^u}+\kappa>\max\{(x^{i^u}: i^u \in\{A^u,B^u,C^u,D^u,E^u,F^u\}\}$. Then the line segment $D^ud^u$, has gradient, $g_{D^ud^u}=\frac{y^{D^u} +\epsilon-y^{D^u}}{x^{D^u}+\kappa-x^{D^u}}=\frac{\epsilon}{\kappa}$. $\un_E$ and $\un_F$ for $\Pi_D^u$ can then be calculated from \eqref{Eq:ValueOfU2D}, taking the limit as $\epsilon \to 0$ gives the result.

Therefore, the minimum amount of uncertainty required for DMUs $E$ and $F$ to become efficient are $\un_E=0.79$ and $\un_F=1.21$. Both first become efficient on the translated version of the line segment $BC$.
\end{example}

The number of possible facets of the efficient frontier increases exponentially as the size of the problem, $M+N$ and $I$, increases,  \citep{briec2003dual}. Therefore, as the problem size increases, explicitly calculating the amount of $\un$ for each hyperplane and then selecting the minimum to find $\minU$, as done in Example \ref{Ex:DEADist2D}, is not possible in polynomial time. There are methods to calculate all the facets, see for example, \citet{frei1999projections,briec2003dual,olesen2015facet} and \citet{ehrgott2019multiobjective}. This leads to the important question of when is it beneficial for an inefficient DMU to compare itself to different facets of the efficient frontier as uncertainty is introduced? 

When $N=M=1$, we can determine in advance to which hyperplane a DMU should be projected. This is because the facets are line segments and can be defined by two extreme points of the efficient frontier (or a single extreme point of $T$ for which an output(input) can be decreased(increased) without entering the interior of the PPS). By solving the nominal DEA problem \eqref{Eq:Nominal} $I$ times, all the efficient DMUs can be determined. There is a finite number, $\phi$, of efficient DMUs whose inputs and outputs represent extreme points of the PPS. These $\phi$ points can be ordered so that $x^1 < x^2< \dots < x^{\phi}$ and $y^1<y^2< \dots < y^{\phi}$ such that each consecutive pair defines a facet of the efficient frontier. This can be used to determine the facet a DMU should be compared to under uncertainty. This obviates the need to calculate the amount of uncertainty for each hyperplane. To do this, we first show in Proposition \ref{Prop:Gradient} that if two DMUs the same DEA distance from the efficient frontier are projected to different facets, the amount of uncertainty required for the facet with a bigger gradient will be greater.
\begin{proposition}\label{Prop:Gradient}
Let $N=M=1$ and assume there are three or more efficient extreme points of the PPS. Let DMU $\ih$ and $\ih'$ be inefficient DMUs a DEA distance $\tau$ from the efficient frontier that are projected to a hyperplane that contains the line segment defined by two extreme points of the PPS with gradients $g$ and $g'$ respectively. If  $g>g'$ then $E^{\ih}<E^{\ih'}$
\end{proposition}
\begin{proof}
Choose three DMUs $A,~B$ and $C$ that are consecutive extreme points of the PPS. Then the efficient frontier contains the line segments ${AB}$ and ${BC}$. We denote their gradients $g$ and $g'$, respectively, where $g>g'$. Consider the DMUs $\ih$ and $\ih'$ and their nominal DEA projections to the efficient frontier, points $D$ and $E$. 
\begin{equation}
\begin{aligned}
\ih&=(x^A+\rho+\tau,y^A+g\rho),~&\ih'=&(x^B+\rho'+\tau,y^B+g'\rho'), \\
D&=(x^A+\rho,y^A+g\rho), &E=&(x^B+\rho',y^B+g'\rho'),
\end{aligned}
\end{equation}
where $\rho,~\rho'$ and $\tau$ are positive constants. In this way DMUs $\ih$ and $\ih'$ have peers $A,~B$ and $B,C$, respectively and $\mathfrak{D}(\ih,AB)=\mathfrak{D}(\ih',BC)=\tau$. From \eqref{Eq:ValueOfU2D},
\begin{equation}
\label{Eq:uDuEALL}
 \un_{\ih}=\frac{g\tau}{2(1+g)} ,~~~
 \un_{\ih'}=\frac{g'\tau}{2(1+g')}.
\end{equation}
But, $g>g'$ so $g$ can be rewritten as $g=g'+\epsilon~(\epsilon>0)$ and $\un_{\ih}$ becomes
 \begin{equation}
 \un_{\ih}= \frac{\tau(g'+\epsilon)}{2(1+g'+\epsilon)}
 	 = \frac{\tau}{2} \left( \frac{g'+\epsilon}{1+g'+\epsilon} \right). 
 \end{equation}
Therefore, $\un_{\ih}>\un_{\ih'}$ and hence the amount of uncertainty required for a DMU to become efficient on a facet is greater when the gradient of the facet is bigger.
\end{proof}
Therefore, for inefficient DMUs a DEA distance $\tau$ from the efficient frontier, as the gradients of the line segments of the efficient frontier decrease, the required uncertainty for an inefficient DMU to become efficient when compared to that line segment decreases. This means the efficiency score of inefficient DMUs will not increase if compared to line segments of the efficient frontier with a larger gradient than the gradient of the line segment a DMU is projected to in the nominal DEA problem \eqref{Eq:Nominal}. Using the result from Proposition \ref{Prop:Gradient}, for $N=M=1$, we can now identify which line segment each point in the PPS is projected to and hence, calculate the minimum uncertainty required for a DMU to be deemed efficient. This is done in Theorem \ref{Thm:Sections}. 

\begin{theorem}\label{Thm:Sections}
Assume $N=M=1$ and there are $\phi>2$ extreme points of the PPS, ordered such that $x^1< \dots < x^{\phi}$ and $y^1< \dots < y^{\phi}$. The line segment of the efficient frontier which requires the minimum amount of uncertainty for DMU $\ih$ to be projected to its corresponding target point can be determined as follows.
\begin{subequations}\label{Eq:WhichSection}
\begin{align*}
  \text{If }  y^{\ih}+x^{\ih}\leq y^1+x^1 &\text{ the line segment is } EF_{1}.\\
  \text{If }  y^1+x^1\leq y^{\ih}+x^{\ih}\leq y^2 +x^2 &\text{ the line segment is } {EF_{1,2}.}\\
        \vdots  ~~~~~~~~~~~~~~~~~~~~~~~~& ~~~~~~~~~~~~~~~~~~~~~~~~\vdots \nonumber\\
  \text{If }  y^{\phi-1} +x^{\phi-1}\leq y^{\ih}+x^{\ih}\leq y^{\phi} +x^{\phi} &\text{ the line segment is } {EF_{\phi-1,\phi}}.\\
  \text{If }  y^{\phi} +x^{\phi}\leq y^{\ih}+x^{\ih}~~~~~~~~~~~~~ &\text{ the line segment is } {EF_\phi}.
\end{align*}
\end{subequations}
Here $EF_{i,j}$ is the straight line formed between $(x^{i^u},y^{i^u})$ and $(x^{j^u},y^{j^u})$, $EF_1$ is the vertical line $x=x^{1^u}$ and $EF_\phi$ is the horizontal line $y=y^{\phi^u}$.
\end{theorem}
\begin{proof}
Let $\un_1<\un_2< \hdots$. From Theorem \ref{Thm:MaxIncUFix} an optimal solution to \eqref{Eq:RobustBoxDEA} for DMU $\ih$ exists such that the uncertain data for DMU $i$ are selected so that as $\un$ increases $(x_i,y_i)$ becomes $(x_i +\un_1, y_i - \un_1), (x_i +\un_2, y_i -\un_2)$ etc.. These points are on the line given by equation  $y=-x+y_i+x_i$. Hence, the introduction of box uncertainty results in the extreme points of the efficient frontier moving towards $(x_i+y_i,0)$ until $\gamma_{\ih}^*=1$. 

First, consider DMU $\ih$ and $\ih'$ where DMU $\ih$ lies to the left of the line $y=-x+y_i+x_i$, DMU $\ih'$ lies to the right and $y_{\phi-1}<y_{\ih}=y_{\ih'}<y_{\phi}$. In this way, $y^{\phi-1} +x^{\phi-1}< y^{\ih}+x^{\ih}< y^{\phi} +x^{\phi}$ and $y^{\phi} +x^{\phi}< y^{\ih'}+x^{\ih'}$ and both have peers DMU $\phi$ and $\phi-1$ in the nominal DEA problem \eqref{Eq:Nominal}. From Proposition \ref{Prop:Gradient} DMU $\ih^u(\ih'^u)$ will be projected to either $EF_\phi$ or $EF_{\phi -1,\phi}$. From \eqref{Eq:ValueOfU2D}, calculate the amount of uncertainty required for DMU $\ih^u$ and $\ih'^u$ to be projected to both $EF_\phi$ and $EF_{\phi -1,\phi}$. This gives $\un_{\ih}^{EF_{\phi -1,\phi}}\leq \un_{\ih}^{{EF_\phi}}$ and $\un_{\ih'}^{EF_{\phi -1,\phi}}\geq \un_{\ih'}^{{EF_\phi}}$, where $\un_{\ih}^{EF_{ij}}$ is the amount of uncertainty required for DMU $\ih$ to be efficient on $EF_{ij}$. If DMU $\ih$ lies on the line $y=-x+y_i+x_i$, then $\un_{\ih}^{EF_{\phi -1,\phi}}= \un_{\ih}^{{EF_\phi}}$. Repeated application to line segments defined by consecutive extreme points of the PPS of \eqref{Eq:Nominal} gives the result. 
\end{proof}
Inefficient DMUs with $y^{\ih}+x^{\ih}<y^1+x^1$ will be compared to the line $x=x^{1^u}$. If a DMU is on the line $x=x^{1^u}$ and has output less than $y^{1^u}$, it would be possible for the DMU to increase its output further without increasing the input. However, it is still deemed efficient as we are considering an input oriented DEA model. When DMU $\ih$ satisfies $ y^{\phi} +x^{\phi}<y^{\ih}+x^{\ih}$, DMU $\ih^u$ will be compared to the line $y=y^{\phi^u}$. Any DMUs on the line $y=y^{\phi^u}$ will require any infinitesimally small amount of uncertainty $\un$ to become efficient, because this will result in $y^{\ih^u}>y^{\phi^u}$.

From Theorem \ref{Thm:Sections}, for $N=M=1$ we can determine the line segment which requires the minimum amount of uncertainty for DMU $\ih^u$ to be projected to. Hence, we can determine the line segment of the efficient frontier on which DMU $\ih^u$ will be deemed efficient. We demonstrate this by deriving the minimum uncertainty for DMUs $E$ and $F$ from Example \ref{Ex:DEADist2D}. 

\addtocounter{example}{-1}
\begin{example}{\textbf{DEA Distance Continued.}}
From Theorem \ref{Thm:Sections}, the line segment of the efficient frontier the inefficient DMUs should be compared to can be calculated. We have $ y^{A}+x^{A}=2,~y^{B} +x^{B}=7,~ y^{C} +x^{C}=14,~y^{D} +x^{D}=18,~y^{E}+x^{E}=13$ and $y^{F}+x^{F}=8$. Therefore, both DMUs $E$ and $F$ should be compared to $BC$ and we only need to calculate $\un_E$ and $\un_F$ when DMUs $E$ and $F$ are compared to $ BC $.
\end{example}

From Theorem \ref{Thm:ValueOfU}, the amount of uncertainty required for DMU $\ih^u$ to be projected to each hyperplane, $\NMu \in \Psi^u$ can be calculated. However, as $N,M$ and $I$ increase in the uDEA problem \eqref{Eq:uDEA}, explicitly calculating each facet and the minimum amount of uncertainty for each DMU to each facet becomes computationally intensive. In Algorithm \ref{Alg:uDEA}, we propose a method to approximate the minimum amount of uncertainty, $\upsilon_{i}^*$, to deem DMU $i$ efficient using the robust DEA method \eqref{Eq:RobustBoxDEA}.

\begin{center}
\begin{algorithm}[H]
\DontPrintSemicolon
\SetAlgoLined
\SetKwInOut{Input}{Input}\SetKwInOut{Output}{Output}
\BlankLine
\Input{$\Omega,~ \max(\map(\U))$, step length $t>0$}
\Output{$\upsilon_{i}^*$ }
\BlankLine
\For{$\ih=1,\ldots I$}{
    Set $\sigma=0$\;
    \While{$\Eih<1$ and $\sigma<\max(\map(\U))$}{
    Transform the data to account for uncertainty \;
    $x^{\ih^u}=x^{\ih}-\sigma,~x^{i^u}=x^i+\sigma,~\forall i \in \mathcal{I},~i \neq \ih$ \;
    $y^{\ih^u}=y^{\ih}+\sigma,~y^{i^u}=y^i-\sigma,~\forall i \in \mathcal{I},~i \neq \ih$\;
    Calculate $\Eih$ \;
    Set $\sigma=\sigma+t$\;
    }
    \eIf{\Eih==1}{
    DMU $\ih$ is capable under $\Omega$ and $\minU=\sigma$.}
    {The $\ih$th DMU is incapable under $\Omega$.}
        }
\caption{General procedure to solve box uDEA \eqref{Eq:uDEA}}
\label{Alg:uDEA}
\end{algorithm} 
\end{center}

In Algorithm \ref{Alg:uDEA}, the step length $t$ is problem dependent. It should be chosen such that only minor gains in efficiency are observed between iterations. The accuracy of Algorithm \ref{Alg:uDEA} at solving \eqref{Eq:uDEA} is limited by the chosen step length $t$, i.e. if for some $t$ DMU $\ih$ becomes efficient, the true minimal amount of uncertainty is within $[\sigma, \sigma-t)$. Therefore, $t$ should be chosen to achieve the desired accuracy of $\minU$. In Section \ref{SubSec:CaseStudy}, we apply this algorithm to determine the minimum amount of uncertainty required for radiotherapy treatment plans to be deemed efficient.

\section{A case study in radiotherapy}\label{SubSec:CaseStudy}

Alongside surgery and chemotherapy, external beam radiotherapy is one of the major forms of cancer treatment and about two thirds of all cancer patients undergo a course of radiotherapy \citep{berkey2010managing}. Radiotherapy exploits a therapeutic advantage whereby cancer cells are unable to recover as well as healthy cells from radiation damage. Radiotherapy treatment is planned with the aim of achieving conflicting goals; while a sufficiently high dose of radiation is necessary for tumour control, a low dose of radiation is desirable to avoid complications in normal, healthy, tissue. Once an initial treatment plan has been made, planners will try and improve it to ensure prescribed doses are met and the Organ At Risk (OAR) tolerances are not exceeded. It is very difficult to design an optimal treatment plan due to the large number of parameters involved. A plan can be evaluated both through visual inspection and quantitative measures to check if it meets the required prescribed doses to the tumour and restrictions on doses to the OARs. For complex cases, it is unlikely to be an acceptable plan the first time round. As a result, the planner must change parameters and replan until a satisfactory plan is found. However, the trial and error aspect means the most desirable plan that could be achieved may not be found. This is particularly difficult in the optimisation processes involved in treatment planning due to the trade-offs between treatment of the tumour and sparing of the OARs. Treating the tumour can be achieved by using very high radiation doses. However, this will also affect healthy cells and subject the patient to high doses of radiation. Conversely, by not irradiating at all, the patient may die of cancer. Neither is desirable. Therefore, a suitable balance must be found between treatment and sparing.

In radiotherapy treatment planning, DMUs are the treatment plans and DEA assesses how well the plans perform in transforming inputs into outputs, i.e. delivering the prescribed dose to the tumour, the planning target volume (PTV), while limiting the dose delivered to OARs. The resulting efficiency score is relative to the set of plans considered in the study. The doses to the OARs, limits which are set by the oncologists, can be thought of as inputs in the DEA model. The dose to the PTV, which ultimately determines the medical outcome of the treatment, is the output in the DEA model. A good treatment plan would minimise the dose to the OARs and maximise the dose to the PTVs. 

Here, we consider 42 anonymised prostate IMRT treatments, 37 of which were used to treat patients at Auckland Radiation Oncology, a private radiation therapy centre in Auckland, New Zealand. The remaining five were plans produced as a result of replanning as described in \citet{lin2013quality}. These plans were also used in \citet{ehrgott2018uncertain} where ellipsoidal uncertainty is considered. After consulting with clinicians, \citet{lin2013quality} decided that the important variables to assess treatment plan quality were the generalised equivalent uniform dose, (gEUD), for the rectum and the $D_{95}$ for the prostate. The gEUD for the rectum is the only input in the DEA model. It is an averaging quantity that measures the homogeneity of the dose delivered to the rectum. The $D_{95}$ for the prostate is the dose (in Gy) received by 95\% of the prostate. They also considered the percentage volume of rectum overlapping the prostate as an environmental variable. This is because although the overlap cannot be changed it can influence the quality of the treatment plan and so should be accounted for. For more details on the use of environmental variables in DEA see \citet{cooper2000discretionary}. After replanning 5 of the plans that they believed could improve, they concluded that ``the results confirm that DEA is capable of identifying plan improvement potential and predicting the best attainable plan in terms of the inputs and outputs". This study assumed the planning data to be exact and classified treatment plans as efficient or inefficient based on these data. However, it is likely that the values for the data are uncertain. Hence, it is also possible that an inefficient plan does actually perform well in practice. This outlook is motivated by the inherently uncertain nature of radiotherapy treatment planning, see for example, \citet{buzdar2013accuracy}.

There are many uncertainties associated with the radiotherapy treatment process. It is well known in radiation oncology that the outcomes of radiotherapy differ from the plans, i.e. the doses delivered to structures are usually slightly different from those calculated during treatment planning. There are many sources of uncertainty in predicting the dose delivered, from the mathematical modelling of the dose distribution and the physical interaction of radiation with the biological tissues, to the variability in human contouring and treatment equipment alignment. In radiotherapy treatment planning the standard assumption is that uncertainty is proportional to the dose. The international commission on radiation units and measurements conclude that the available evidence for certain types of tumour suggests an accuracy of $\pm5\%$ is required \citep{andreo2004commissioning}. Combining the standard uncertainty value for dose determination and the uncertainty associated with the treatment planning system used, here Pinnacle \citep{Pinnacle}, \citet{Henriquez2008Uncertainty} suggest an uncertainty of $3.6\%$ is used. Therefore, we use $\max (\map (\U))=3.6\%$, i.e. if a plan is not efficient with 3.6\% or less uncertainty it is deemed incapable.

We wish to consider the effect uncertainty has on the conclusions drawn in \citet{lin2013quality}. We first transform the data to aid ease of computation. We divide the output $D_{95}$ by the prescribed dose $74 Gy\times 0.95$. Then our output is a measure of the proportion of the prescribed dose achieved by 95\% of the PTV volume. Here we note that although proportion variables generally cause problems in DEA models, this is not the case for our output variable as it is re-normalisation with the same denominator for each output, \citep{olesen2015efficiency}. Similarly, we divide the input gEUD to the rectum by 70 as clinically there is a higher risk of toxicity related to rectal doses over 70Gy \citep{tucker2012intermediate}. Multiplying both by 100 means we can model the uncertainty of up to $3.6\%$ uncertainty in both the input and output by using $\un=3.6$ directly. 

\citet{lin2013quality} include the percentage volume of the rectum overlapping the PTV to account for anatomical variations between patients. This is incorporated in the DEA model via a non-discretionary output variable \citep{cooper2000discretionary}. Here we will consider the DEA model both with and without the overlap included. Note that we do not consider uncertainty for the percentage volume of rectum overlapping the PTV. Instead, we consider that it has the same effect on treatment plan quality for each realisation of the data. 

From Theorem \ref{Thm:MaxIncUFix}, we can solve the robust DEA model \eqref{Eq:RobustBoxDEA} for uncertainty in increasing values of $\un$. We consider an increase in uncertainty up to the clinically relevant $\un=3.6\%$. Figure \ref{Fig:IMRTuDEA} shows the nominal efficiency score plotted against the minimum amount of uncertainty required to deem a DMU efficient, where the value of uncertainty has been increased in step sizes of $t=0.01$. We use $t=0.01$ as it is a sensible step size for the application and allows us to report the minimum amount of uncertainty correctly to two decimal places.

In Figure \ref{Fig:IMRTuDEA1}, the plans that are efficient when overlap is included are shown by a blue star symbol. These plans are all at $(1,0)$ in Figure \ref{Fig:IMRTuDEA2} where we include overlap as an environmental variable. 
The introduction of a non-discretionary output variable means different plans will define the efficient frontier depending on the amount of overlap being considered. Here, the plots are colour-coded to demonstrate the effect overlap has on our DEA results. 
\begin{figure}
\begin{subfigure}[t]{0.5\textwidth}
        \centering
      {\includegraphics[scale=0.4]{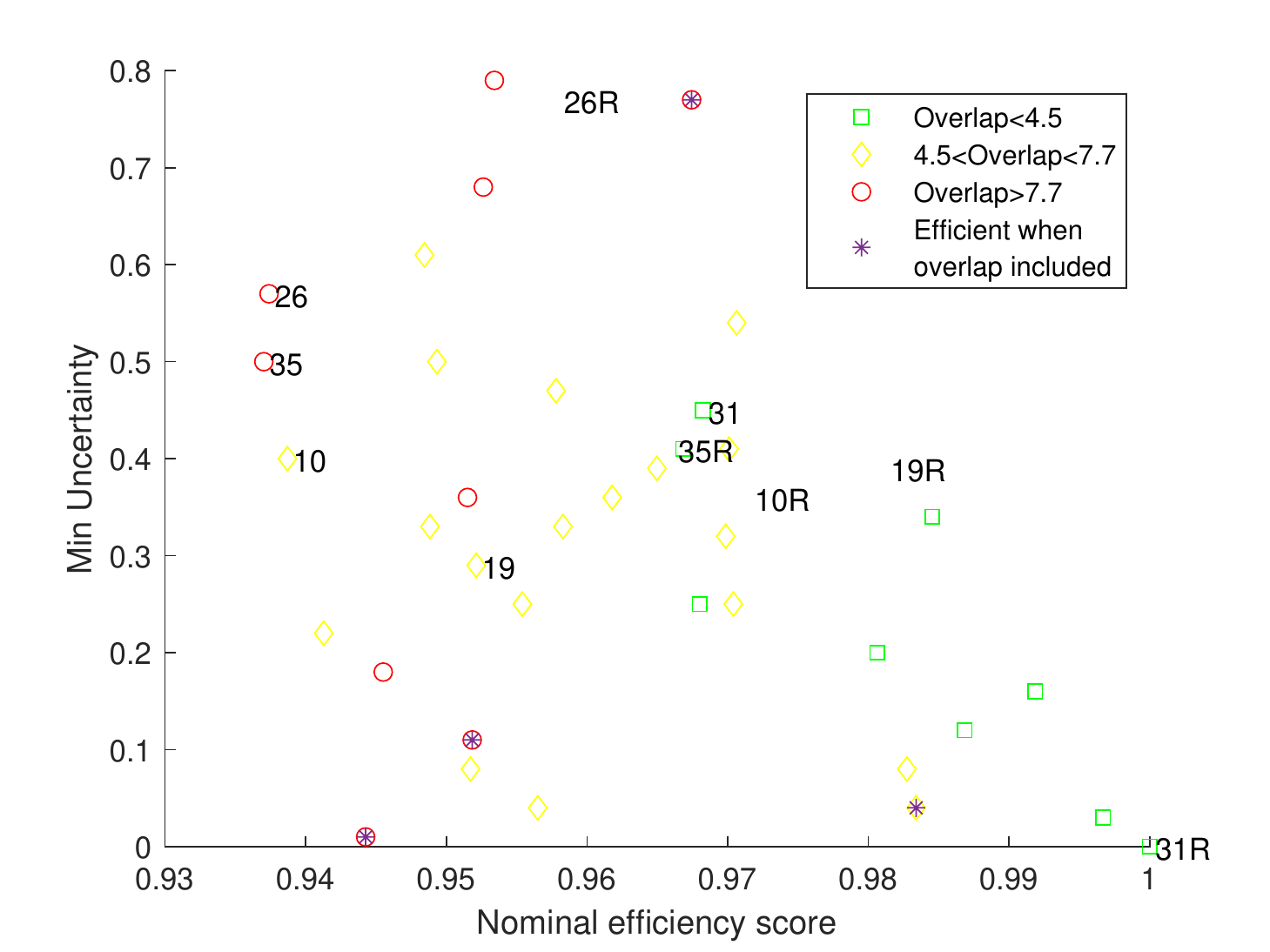}}
        \caption{Without overlap}
        \label{Fig:IMRTuDEA1}%
    \end{subfigure}%
    \begin{subfigure}[t]{0.5\textwidth}
        \centering
        {\includegraphics[scale=0.4]{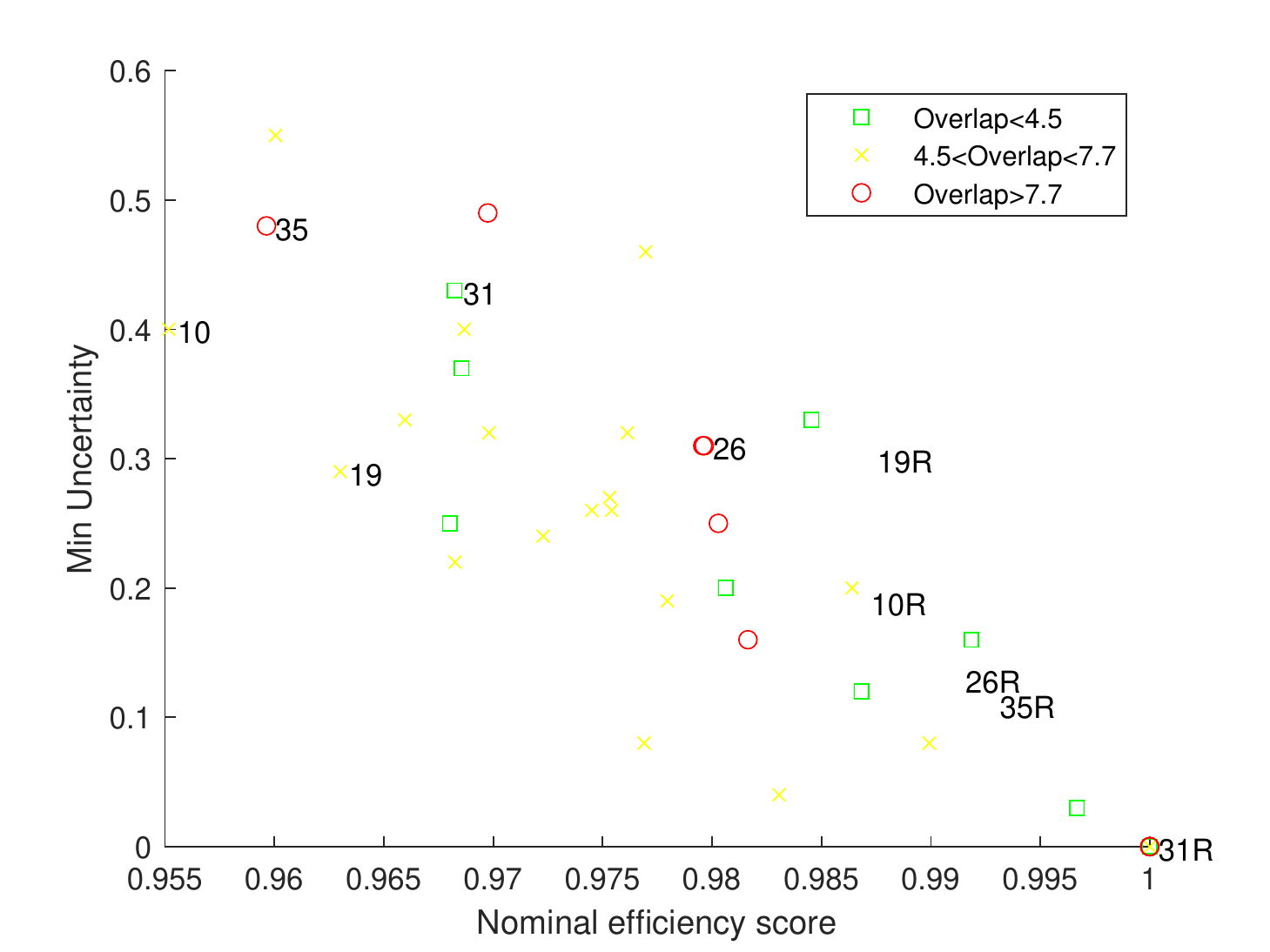}}
        \caption{With overlap}
        \label{Fig:IMRTuDEA2}%
    \end{subfigure}
    \caption{Results of the uDEA problem for IMRT data.}
    \label{Fig:IMRTuDEA}%
\end{figure}
Here, we note that for plan 19, although the efficiency score increased with replanning, the amount of uncertainty required for it to be evaluated efficient increased. This highlights the importance of considering the role uncertainty plays when choosing plans to be improved.

We argue that it may be more beneficial to replan plans that require a large amount of uncertainty to be deemed efficient, even if they have a better nominal efficiency score. This is because these are the plans that cannot argue their inefficiency is due to the uncertain data.

\section{Conclusion}

In this paper, we developed a method for solving the uDEA model \eqref{Eq:uDEA} for box uncertainty. If the efficient facets of the PPS are known the uDEA problem can be solved exactly. This can be done by considering the amount of uncertainty required for the uncertain inputs and outputs of a DMU to be projected to each efficient facet. But, as $M+N$ and $I$ increase this is not possible in polynomial time. Therefore, we have proposed an algorithm (Algorithm \ref{Alg:uDEA}) that iteratively increases the amount of uncertainty $\un$ allowed, solving a robust DEA model by linear programming in each iteration. Clearly, methods that do not rely on the full specification of all efficient facets would be beneficial to progress in this field. This algorithm computes the amount of uncertainty defined by \eqref{Eq:uDEA} up to a deviation equal to the iteration step length.

This paper has built on work by \citet{ehrgott2018uncertain} where a first order algorithm for solving the uDEA problem with ellipsoidal uncertainty was provided. Because uDEA problems are non-convex, further research into heuristic approaches may be beneficial. The uDEA model involves nonlinear terms so an alternative approach would be to determine a suitable simplified model. A model that captures the properties of the uDEA model by introducing binary variables to linearise the nonlinear constraints. This is an area we will explore further.
\subsection*{Acknowledgement}
We gratefully acknowledge the support of the EPSRC funded EP/L015692/1 STOR-i Centre for Doctoral Training.
\bibliographystyle{apalike}  


\bibliography{bib}

\begin{thebibliography}{}

\bibitem[Andreo et~al., 2004]{andreo2004commissioning}
Andreo, P., Cramb, J., Fraass, B., Ionescu-Farca, F., Izewska, J., Levin, V.,
  Mijnheer, B., Rosenwald, J., Scalliet, P., Shortt, K., et~al. (2004).
\newblock Commissioning and quality assurance of computerized planning systems
  for radiation treatment of cancer.
\newblock Technical Report 430, International Atomic Energy Agency.

\bibitem[Berkey, 2010]{berkey2010managing}
Berkey, F.~J. (2010).
\newblock Managing the adverse effects of radiation therapy.
\newblock {\em American Family pPysician}, 82(4):381--388.

\bibitem[Briec and Leleu, 2003]{briec2003dual}
Briec, W. and Leleu, H. (2003).
\newblock Dual representations of non-parametric technologies and measurement
  of technical efficiency.
\newblock {\em Journal of Productivity Analysis}, 20(1):71--96.

\bibitem[Buzdar et~al., 2013]{buzdar2013accuracy}
Buzdar, S.~A., Afzal, M., Nazir, A., and Gadhi, M.~A. (2013).
\newblock Accuracy requirements in radiotherapy treatment planning.
\newblock {\em Journal of the College of Physicians and Surgeons Pakistan},
  23(6):418--23.

\bibitem[Charnes and Cooper, 1984]{charnes1984preface}
Charnes, A. and Cooper, W.~W. (1984).
\newblock Preface to topics in data envelopment analysis.
\newblock {\em Annals of Operations Research}, 2(1):59--94.

\bibitem[Cooper et~al., 2000]{cooper2000discretionary}
Cooper, W., Seiford, L., and Tone, K. (2000).
\newblock Discretionary, non-discretionary and categorical variables.
\newblock {\em Data Envelopment Analysis: A Comprehensive Text with Models,
  Applications, References and DEA-Solver Software}, pages 183--219.

\bibitem[Ehrgott et~al., 2019]{ehrgott2019multiobjective}
Ehrgott, M., Hasannasab, M., and Raith, A. (2019).
\newblock A multiobjective optimization approach to compute the efficient
  frontier in data envelopment analysis.
\newblock {\em Journal of Multi-Criteria Decision Analysis}, 26(3-4):187--198.

\bibitem[Ehrgott et~al., 2018]{ehrgott2018uncertain}
Ehrgott, M., Holder, A., and Nohadani, O. (2018).
\newblock Uncertain data envelopment analysis.
\newblock {\em European Journal of Operational Research}, 268(1):231--242.

\bibitem[Frei and Harker, 1999]{frei1999projections}
Frei, F. and Harker, P. (1999).
\newblock Projections onto efficient frontiers: Theoretical and computational
  extensions to {DEA}.
\newblock {\em Journal of Productivity Analysis}, 11(3):275--300.

\bibitem[Gorissen et~al., 2015]{gorissen2015practical}
Gorissen, B.~L., Yan{\i}ko{\u{g}}lu, {\.I}., and den Hertog, D. (2015).
\newblock A practical guide to robust optimization.
\newblock {\em Omega}, 53:124--137.

\bibitem[Henr{\'\i}quez and Castrill{\'o}n, 2008]{Henriquez2008Uncertainty}
Henr{\'\i}quez, F. and Castrill{\'o}n, S. (2008).
\newblock The effect of the different uncertainty models in dose expected
  volume histogram computation.
\newblock {\em Australasian Physics {\&} Engineering Sciences in Medicine},
  31(3):196--202.

\bibitem[Lin et~al., 2013]{lin2013quality}
Lin, K.-M., Simpson, J., Sasso, G., Raith, A., and Ehrgott, M. (2013).
\newblock Quality assessment for {VMAT} prostate radiotherapy planning based on
  {DEA}.
\newblock {\em Physics in Medicine {\&} Biology}, 58(16):5753.

\bibitem[Olesen and Petersen, 2015]{olesen2015facet}
Olesen, O. and Petersen, N. (2015).
\newblock Facet analysis in data envelopment analysis.
\newblock In Zhu, J., editor, {\em Data Envelopment Analysis -- A Handbook of
  Models and Methods}, pages 145--190. Springer.

\bibitem[Olesen et~al., 2015]{olesen2015efficiency}
Olesen, O.~B., Petersen, N.~C., and Podinovski, V.~V. (2015).
\newblock Efficiency analysis with ratio measures.
\newblock {\em European Journal of Operational Research}, 245(2):446--462.

\bibitem[Philips, 2009]{Pinnacle}
Philips (2009).
\newblock Pinnacle treatment planning.
\newblock
  \url{https://www.philips.co.uk/healthcare/solutions/radiation-oncology/radiation-treatment-planning}.

\bibitem[Thanassoulis and Allen, 1998]{thanassoulis1998simulating}
Thanassoulis, E. and Allen, R. (1998).
\newblock Simulating weights restrictions in data envelopment analysis by means
  of unobserved {DMU}s.
\newblock {\em Management Science}, 44(4):586--594.

\bibitem[Tucker et~al., 2012]{tucker2012intermediate}
Tucker, S., Dong, L., Michalski, J., Bosch, W., Winter, K., Cox, J., Purdy, J.,
  and Mohan, R. (2012).
\newblock Do intermediate radiation doses contribute to late rectal toxicity?
  {A}n analysis of data from radiation therapy oncology group protocol 94-06.
\newblock {\em International Journal of Radiation Oncology* Biology* Physics},
  84(2):390--395.

\end{thebibliography}

\end{document}